\documentclass[10 pt]{amsart}

\usepackage{amsmath} 
\usepackage{amsfonts}
\usepackage{amssymb}
\usepackage{amstext}
\usepackage{amsbsy}
\usepackage{amsopn}
\usepackage{amsthm}
\usepackage{amsxtra}
\usepackage{graphicx}
\usepackage{hyperref}
\usepackage{color}

\newtheorem{theorem}{Theorem}[section]
\newtheorem{lemma}[theorem]{Lemma}
\newtheorem{proposition}[theorem]{Proposition}
\newtheorem{corollary}[theorem]{Corollary}
\theoremstyle{definition}
\newtheorem{definition}[theorem]{Definition}
\newtheorem{question}[theorem]{Question}
\newtheorem{example}[theorem]{Example}
\newtheorem{notation}[theorem]{Notation}

\newcommand{\R}{\mathcal{R}}
\newcommand{\Z}{\mathbb{Z}}

\newcommand{\N}{\mathbb{N}}

\newcommand{\CA}{\mathcal{A}}
\newcommand{\CM}{\mathcal{M}}
\newcommand{\CF}{\mathcal{F}}
\newcommand{\CL}{\mathcal{L}}

\def\ts{{\texttt s}}
\def\tp{{\texttt p}}
\def\tf{{\texttt f}}
\def\tu{{\texttt u}}

\newcommand{\Aut}{\mathrm{Aut}}

\newcommand{\ttp}{\mathrm{top}}
\newcommand{\per}{\mathrm{Per}}

\title[Strong Approximations of Shifts]{Strong Approximations of Shifts and the Characteristic Measures Problem}

\author{Van Cyr}
\address{Bucknell University, Lewisburg, PA 17837 USA}
\email{van.cyr@bucknell.edu}
\author{Bryna Kra}
\address{Northwestern University, Evanston, IL 60208 USA}
\email{kra@math.northwestern.edu}
\author{Samuel Petite}
\address{Universit\'e de Picardie Jules Verne, Amiens, 80039 France}
\email{samuel.petite@u-picardie.fr}

\subjclass[2010]{37B10, 68R15, 37A15}
\keywords{subshift, block complexity, automorphism, invariant measure}

\thanks{ B. Kra acknowledges the support of National Science Foundation grant DMS-205464 and S. Petite  that of ANR IZES-ANR-22-CE40-0011.}
\begin{document}
\maketitle
\begin{abstract}
Every symbolic system supports a Borel measure that is invariant under the shift, but it is not known if every such systems supports a measure that is invariant under all of its automorphisms; known as a characteristic measure.  We give sufficient conditions to find a characteristic measure, additionally showing when it can be taken to be a measure of maximal entropy.  The class of systems to which these sufficient conditions apply is large, containing a dense $G_{\delta}$ set in the space of all shifts on a given alphabet, and is also large in the sense that it is closed under taking factors.  We also investigate natural systems to which these sufficient conditions apply.
\end{abstract}
\section{Introduction}
By the Kryloff-Bogoliouboff Theorem~\cite{KB}, any topological dynamical systems $(X,T)$ supports a $T$-invariant Borel probability measure.  A natural question is when such a system supports a measure, called characteristic by Frisch and Tamuz~\cite{FT2},  that is invariant under the symmetries associated to the system, namely invariant under every automorphism of the system.  
We note that this question is topological in nature.  Given a topological system and a given Borel measure, the property of being a characteristic measure is not a measure theoretic isomorphism invariant: the Jewett-Krieger Theorem implies that any ergodic measure preserving system is measure theoretically isomorphic to a uniquely ergodic topological system, and hence to a system having a characteristic measure.

Frisch and Tamuz~\cite{FT2} proved that every zero entropy symbolic system (on a finite alphabet) has such a measure, while also observing that not every topological dynamical system does.  Specifically, the identity map acting on the Cantor set does not.  Symbolic systems, however, have special features not shared by all topological systems and results in the literature show that symbolic systems often have a characteristic measure: any mixing shift of finite type does~\cite{Parry}, any shift with zero entropy does~\cite{FT2}, and any shift that contains periodic points supports such a measure as well.  With these observations in mind, Frisch and Tamuz asked if every symbolic system has a characteristic measure~\cite{FT2}.  The goal of the present work is to make progress on this problem.

In an approach to the characteristic measures problem, a new class of shifts, the language stable shifts, is defined in~\cite{CK21}, where it was shown that each such shift supports a characteristic measure of maximal entropy.  The class of language stable shifts is large in the sense that it forms a dense $G_\delta$ in the space of all shifts on a fixed alphabet, as well as in the subspace of shifts whose entropy is bounded from below by some $h\geq 0$.  It also contains all shifts of finite type.  One might ask for other senses in which the class of language stable shifts is large and, specifically we consider if this class is closed under taking factors, and a first result proved in Section~\ref{sec:strictly-sofic} shows that it is not. 

\begin{theorem}
\label{th:sofic-not-LSS}
Any shift that is both language stable and is sofic is a shift of finite type. 
\end{theorem}

Thus any sofic shift that is not a shift of finite type is not language stable, but is a factor of a language stable shift.
This motivates our introduction in Section~\ref{sec:well-approx} of a class of shifts that are  sufficiently well approximated by shifts of finite type in order that properties of characteristic measures are preserved. 
 
Returning to the characteristic measures question we note that sofic shifts are known to have such measures for other reasons (for example, because they have periodic points or, in the irreducible case, because they are intrinsically ergodic).  But these properties are not shared by all factors of language stable shifts, and even language stable shifts themselves need not have periodic points or only finitely many ergodic measures of maximal entropy.  This naturally leads us to the question of whether every subshift factor of a language stable shift has a characteristic measure and whether this measure can be taken to be a measure of maximal entropy.  To investigate these questions, in Section~\ref{sec:well-approx} we define a notion that we call  well-approximable language stable shifts, quantifying how well the shift can be approximated by shifts of finite type, loosely analogous to how well a real number can be approximated by rationals of a given denominator.  We show that well-approximable language stable shifts remain a large set, in the sense that they continue to contain a dense $G_{\delta}$ in the space of all shifts (see Proposition~\ref{prop:g-delta}), and that symbolic factors of such shifts have characteristic measures of maximal entropy (see Theorem~\ref{th:char-periodic} for the precise statement).
\begin{theorem}
Every subshift factor of a well-approximable language stable shift has a characteristic measure that is a measure of maximal entropy.
\end{theorem}
We give an effective rate of approximation by subshifts of finite type that guarantees the existence of a characteristic measure (but not necessarily a characteristic measure of maximal entropy) in Theorem~\ref{th:char-periodic}.  We also give a faster (and ineffective) rate of approximation that guarantees the existence of a characteristic measure of maximal entropy in Proposition~\ref{prop:well-approx-factors}.

If every shift were a factor of a well-approximable language stable shift, we would then be able to conclude that every symbolic system supports a characteristic measure.  Although we do not rule this possibility out entirely, in Section~\ref{sec:not-all-wa}, we construct an example of a shift that is not a factor of any sufficiently well-approximable language stable shift. However, we do not answer the question if every subshift factor of a language stable shift has a characteristic measure.

We conclude in Section~\ref{sec:large-LSS} by showing that the class of language stable shifts is closed under several natural operations, such as passing to a power of the shift or more generally passing to any speed-up.  
We continue the section by exhibiting language stable shifts in other well-studied classes of shifts, and include a characterization of which $\beta$-shifts are language stable.  In Section~\ref{sec:linear}, we show that any aperiodic linear complexity shift is language stable.

This leaves open numerous questions about language stable shifts and their characterizations, including Question~\ref{question:complexity-theshold}  on the complexity threshold for a shift that is not language stable and questions in Section~\ref{sec:questions-on-WA} on quantitative bounds for well-approximable shifts.

\section{Classes of shifts and their properties}

\subsection{Symbolic systems and subshifts}
Assume that $\CA$ is a finite set, the {\em alphabet}, and we denote $x\in \CA^\Z$  by $x = (x_n)_{n\in\Z}$.  
The space $\CA^\Z$ is a compact metric space when endowed with the metric $d$ defined by 
$$d\bigl((x_n), (y_n)\bigr) = 2^{-\inf\{|n|: x_n\neq y_n\}}.$$
The {\em left shift} $\sigma\colon \CA^\Z\to \CA^Z$ defined as $(\sigma x)_n = x_{n+1}$ for all $n\in\Z$ is a homeomorphism of $\CA^\Z$.  If $X\subset\CA^\Z$ is a (non-empty) closed and $\sigma$-invariant, then the pair $(X, \sigma)$ is a {\em shift}, sometimes also referred to as a {\em subshift}.  We also sometimes omit the transformation $\sigma$ from the notation and refer to $X$ as a shift. 

The {\em cylinder sets}, meaning the sets where finitely many coordinates are determined, form a basis for the topology on the space $X$. 

\subsection{Presenting a shift via its language}
\label{sec:language}

If $\CA$ is a finite alphabet, $\CA^*$ denotes all finite words in the alphabet $\CA$. The concatenation of two words $u,v \in \CA^*$ is denoted by $uv$ and the word  $u$ is called a \emph{prefix} of $uv$ and the word $v$ is a \emph{suffix} of $uv$.  If $u$ is nonempty, then we say that it is a \emph{strict} prefix, and similarly when $v$ is nonempty it is a \emph{strict} suffix. We say that the word $v$ \emph{occurs} in the finite word $uv$ and  
use the same terminology in infinite words. For $u\in\CA^*$, we let $|u|$ denote the length of the word $u$,  and for $n\geq 1$, the word $u^n$  denotes the word of length $n|u|$ obtained by concatenating $u$  with itself $n$ times.

If $\CF \subset  \CA^*$, then the shift $X_\CF$ associated to the collection $\CF$ of \emph{forbidden words}  is defined by
$$
X_\CF= \{x \in \CA^\Z\colon  \textrm{ no subword of }x \textrm{ belongs to } \CF  \}.
$$
Conversely, any shift $X$ is defined  by a collection of forbidden words: if 
$$ 
\mathcal{F}(X):=\{w\in\CA^*\colon w\text{ does not occur in any element of }X\} 
$$ 
then $X = X_{\CF(X)}$.

For the shift $(X, \sigma)$, the set of all allowable words is the \emph{language} of the shift and is denoted by $\CL(X)$.  
The word $w\in\CL(X)$ is \emph{left special} if there exist distinct letters $a, b\in\CA$ such that both $aw, bw\in\CL(X)$, and similarly if 
there exist distinct letters $a, b\in\CA$ such that $wa, wb\in\CL(X)$, we say that it is \emph{right special}. 
A word is \emph{bispecial} if it is both left and right special.

We write $\CL_n(X)$ for the words of length $n$ in the language $\CL(X)$, meaning that $\CL_n(X) = \CL(X)\cap \CA^n$.  The {\em complexity} of a shift $X$ is defined to be  function  $p_X\colon\N\to\N$ given by $p_X(n)= | \CL_n(X)|$ 
and the exponential growth rate of $p_X(n)$ is the {\em (topological) entropy} of the shift $(X, \sigma)$ and is denoted by $h_{\ttp}(X)$.

Any word $w\notin\mathcal{L}(X)$ is said to be \emph{not allowed} in $X$.  When the set $\CF$ of forbidden words 
is finite, the shift $(X_{\CF}, \sigma)$ is known as a \emph{shift of finite type}.

Given a shift $(X, \sigma)$, define $\mathcal{F}_n:=\mathcal{F}(X)\cap\CA^n$ (the forbidden words in $X$ with length at most $n$) and let $X_n:=X_{\mathcal{F}_n}$ be the shift of finite type defined by the set of forbidden words $\mathcal{F}_n$.  Then we can write $X = \bigcap_{n=1}^\infty X_n$ as the intersection of a canonical, descending chain of nested shifts of finite type and we call the sequence of shifts $(X_n)_{n\in\N}$ the {\em SFT cover} of X.  Note that $X$ is a shift of finite type if and only if $X=X_n$ for all sufficiently large $n\in\N$.

For a word $w\in\CL(X)$, we let $[w]$ denote the cylinder set determined by the word $w = w_1\cdots w_{|w|}$, meaning the collection of all $x\in X$ such that $x_i = w_i$ for all $1\leq i\leq|w|$.

\subsection{Language stable shifts}
\label{sec:LSS}
We recall a notion introduced in \cite{BMR}. If $X$ is a shift with forbidden words $\CF(X)$,  let   
$$
\CM(X) =\{w \in \CF(X) \colon \textrm{no proper subword of } w \textrm{ lies in } \CF(X) \}
$$
denote the \emph{minimal forbidden words} in the language $\CL(X)$ of the shift $X$. 
We note that  $\CL(X) = \CA^* \setminus \CA^* \CM(X)  \CA^*$,  meaning that the allowed words are all the words that do not contain a minimal forbidden word.  So   the sets $\CF(X)$ and $\CM(X)$  define the same shift, meaning that $X_{\CF(X)} = X_{\CM(X)}$.

We note that if $a,b\in \CA$ and the word $aub$ is a minimal forbidden word  of $X$, then the word $u$ is bispecial (see~\cite[Remark 3]{BMR}).

For each $n>0$, set  $\CM_n(X) =\CM(X) \cap \CA^n $ to be the set of minimal forbidden word of length $n$.
A shift $(X,\sigma)$ is \emph{language stable} if the set of lengths
\begin{equation}
\label{def:LSS2}
 LS(X) = \{ n \in \N\colon \CM_n(X) \neq \emptyset\}
\end{equation}
has zero lower uniform density, meaning that 
\[\lim_{k\to\infty} \inf_{n\in\N} \frac{|LS(X)\cap [n+1,n+k]|}{k}  =0.
\]
Equivalently, if $X$ is a shift and $(X_n)_{n\in\N}$ is its SFT cover, then $X$ is language stable if and only if for all $k\in\N$ there exists $n\in\N$ with $X_n=X_{n+k}$.

We introduce a new definition to quantify the relationship between the parameters $n$ and $k$ in the definition of language stable.

\begin{definition}\label{def:well-approx}
Suppose $\alpha\colon\N\to\N$ is non-decreasing and unbounded.  A shift $X$, with SFT cover $(X_n)_{n\in\N}$ is {\em well-approximable at rate $\alpha$} if there are infinitely many $n$ such that $X_n=X_{n+\alpha(n)}$.
\end{definition}

Thus the set of language stable shifts is the union over all non-decreasing and unbounded functions $\alpha$ 
of the shifts that are well-approximable at rate $\alpha$.  In theorems about well-approximable shifts, we are explicit about the rate $\alpha$, but when speaking more informally about results we  sometimes just call a shift well-approximable with the understanding that it is well-approximable at some (typically fast) rate.

We include an example showing that this notion of well-approximable  is not vacuous, as up to finitely many elements,  any set of lengths $LS(X)$ arises. 

\begin{example}
\label{ex:nonempty}
For any set $L\subset \N$, there is a subshift $X$ such that the set of lengths of its minimal forbidden words $L S(X)$  is all but finitely many elements of  $L$. To check this, it suffices to consider the subshift on 3 letters $\{0, 1, 2\}$ with words of the form $\{ 0 1^{n}0: n \in L\cap 2\N\} \cup \{0 2^{n}0:  n \in L \cap (2\N+1)\}$ forbidden. 
Since the words of the form $01^{n}0$ for odd $n$ are allowed, such words with even $n$ are minimal forbidden words;  
the analogous property holds for words of the form $02^{n}0$.
\end{example} 

\subsection{Topological properties of shifts and automorphisms}
The shift $(Y, \sigma)$ is a {\em factor} of the shift $(X, \sigma)$ if there exists a surjective and continuous map $\pi\colon X\to Y$ satisfying $\pi\circ \sigma = \sigma\circ\pi$, and if additionally the map $\pi$ is injective, we say that the systems are {\em conjugate}. 
A shift is {\em sofic} if it is a factor of a shift of finite type. 

The shift $(X,\sigma)$ is {\em mixing} if for all words $u, v\in\CL(X)$, there exists some $N\in\N$ such that for all $n\geq N$, there is some $w\in\CL_n(X)$ such that $uwv\in\CL(X)$.

An {\em automorphism} of the shift $(X, \sigma)$ is a self conjugacy of the system and we denote the group of all automorphisms of the system by $\Aut(X, \sigma)$, or just $\Aut(X)$ when the shift is clear from the context.   The classic theorem of Curtis, Hedlund, and Lyndon (see~\cite{hedlund}) states that any factor $\phi$ of a shift $(X, \sigma)$ is a block code of some range $R$ for some $R\in\N$, meaning that for all $x\in X$, the symbol that $\phi(x)$ assigns to $0$ is determined by the word $x_{-R}\cdots x_R$ in $x = (x_n)_{n\in\Z}$.  
A block code of range $R$ is a block code of range $S$ for any $S> R$.  The same holds for any automorphism, and without loss we assume throughout that the range of any block code arising from an automorphism $\phi$ is also the range for $\phi^{-1}$, and we refer to such a range as being {\em symmetric}. 

\subsection{Invariant measures}
By the Kryloff-Bogoliouboff Theorem~\cite{KB}, there is always an invariant measure on a shift $(X, \sigma)$. If the entropy of this measure is equal to the topological entropy of the system, then this is a {\em measure of maximal entropy}.  

A measure that is invariant under the full automorphism group $\Aut(X, \sigma)$ is said to be a {\em characteristic measure}.

\section{Strictly sofic shifts are not language stable}
\label{sec:strictly-sofic}
It is shown in~\cite{CK21} that every language stable shift has a characteristic measure of maximal entropy.  This motivates the question of which shifts are not language stable, 
as these are the shifts for which the characteristic measures problem remains open.  Our first result shows that the familiar class of sofic shifts is comprised almost entirely of shifts that are not language stable.  
More precisely, we show that the intersection of the classes of sofic shifts and language stable shifts is exactly the shifts of finite type. 
  
\begin{proof}[Proof of Theorem~\ref{th:sofic-not-LSS}]
Let $\CF\subset\CA^*$ be the list of forbidden words defining the sofic shift $(X, \sigma)$. 
 Since  $(X, \sigma)$ is sofic, 
 the set  $\CF$ is rational, meaning that the collection is recognized by a finite state deterministic  automaton.    
For contradiction, assume that $(X, \sigma)$ is not a shift of finite type.  
Then for any $\ell >0$, there is some
word $w \in \CF$ of length at least $\ell+1$ 
that does not contain some word in $\CF$ as a subword.  Thus we can pick some such word $w$ 
whose length is greater than the number of states of the finite state deterministic automaton. 
The word $w$ corresponds to  some vertex  path $v_0, \ldots, v_\ell$, starting with some initial vertex $v_0$ and ending with the vertex $v_\ell$, and by choice of the length of the word it passes at least 
twice through some vertex.  Thus there exist $i\neq j$ such that $v_i= v_j$. 
However, no $v_{k}$ is an accepting state of the automaton, since the word $v_0\cdots v_{k-1}$ is not one of the forbidden words in the collection $\CF$.  Similarly, the path associated to the word $v_1\cdots v_k$ does not end in an accepting state.  
It follows that  for any $n\ge 0$, 
the edge path associated the path $v_0, \ldots,  v_{i-1}, (v_i, \ldots, v_j)^n , v_{j+1}, \ldots,  v_\ell$ belongs to $\CF$ but no subword belongs to $\CF$. In other words, the lengths of minimal forbidden words has positive density and so the shift $(X, \sigma)$ is not language stable. 
\end{proof}

\section{Characteristic measures for well-approximable language stable shifts}
\label{sec:well-approx}

\subsection{Obtaining a characteristic measure using periodic points}
The goal of this section is to show that factors of well-approximable shifts have characteristic measures, provided the shifts are well-approximable at a sufficiently fast rate. 

\begin{definition}
Let $a,f\in\N$ and let $\mathcal{X}_{a,f}$ denote the set of all subshifts of finite type $X\subseteq\{0,1,\dots,a-1\}^{\Z}$ 
that can be defined using only forbidden words of length at most $f$.  Let $R\in\N$ and let $\mathcal{Y}_{a,f,R}$ 
denote the set of all sofic shifts $Y\subseteq\{0,1,\dots,a-1\}^{\Z}$ for which there exist $X\in\mathcal{X}_{a,f}$ and a  factor map $\varphi\colon X\to Y$ that can be implemented using a block code of range at most $R$.
\end{definition}

It follows immediately from the definitions that up to renaming the letters of the alphabet, every shift of finite type is an element of $\mathcal{X}_{a,f}$ for some choice of $a, f\in\N$.  We check that the same holds for any sofic shift.  

\begin{lemma}\label{lem:sofic}
Up to renaming the letters of the alphabet, for any sofic shift $Y$ there exist $a,f,R\in\N$ such that $Y\in\mathcal{Y}_{a,f,R}$.
\end{lemma}
\begin{proof}
By renaming the letters of the alphabet if necessary, we can assume that $Y\subseteq\{0,1,\dots,a_Y-1\}^{\Z}$
for some $a_Y\in\N$.  Since $Y$ is sofic, there exists a shift of finite type $X$ and a factor map $\varphi\colon X\to Y$.  Fixing some finite set of forbidden words such that $X$ is the 
shift of finite type obtained by forbidding these words, 
define $f$ to be the maximal length of the words in this set.  
Again renaming the letters of the alphabet if necessary, 
we can choose $a_X\in\N$ such that $X\subseteq\{0,1,\dots,a_X-1\}^{\Z}$.  
Setting $a:=\max\{a_X,a_Y\}$, we have that $X,Y\subseteq\{0,1,\dots,a-1\}^{\Z}$.  Then, by the Curtis-Hedlund-Lyndon Theorem, there exists $R\in\N$ such that $\varphi$ can be implemented by a range $R$ block code 
and so  $Y\in\mathcal{Y}_{a,f,R}$.  
\end{proof}

Note that $\mathcal{Y}_{a,f,R}\subseteq\mathcal{Y}_{a^{\prime},f^{\prime},R^{\prime}}$ provided $a\leq a^{\prime}$, $f\leq f^{\prime}$, and $R\leq R^{\prime}$.  Combining this and Lemma~\ref{lem:sofic}, it follows that the set $\bigcup_{n=1}^{\infty}\mathcal{Y}_{n,n,n}$ contains an isomorphic copy of every sofic shift (where the isomorphism only changes the names of the letters in the alphabet).

\begin{notation}
For a subshift $Y$ and $p\in\N$, let 
$$ 
\per_p(Y)=\{y\in Y\colon\sigma^py=y\} 
$$ 
denote the set of all periodic points in $Y$ of (not necessarily minimal) period $p$.  Let  
$$ 
\per_{\leq p}(Y):=\bigcup_{q=1}^p\per_q(Y) 
$$ 
denote the set of all periodic points in $Y$ of minimal period at most $p$.
\end{notation}

We check that if two sofic shifts have the same language up to some large scale, then they have the same periodic points of small period.

\begin{lemma}\label{lem:language_periodic}
Fix $p\in\N$.  Suppose that $Y_1,Y_2\subseteq\mathcal{Y}_{a,f,R}$ for some $a, f, R\in\N$ and that there exists $N>p\cdot a^{\max\{f,2R+1\}}$ such that $\mathcal{L}_N(Y_1)=\mathcal{L}_N(Y_2)$.  Then $\per_{\leq p}(Y_1)=\per_{\leq p}(Y_2)$. 
\end{lemma}
\begin{proof}
For $i=1,2$, choose shifts of finite type $X_i\in\mathcal{X}_{a,f}$ and factor maps $\varphi_i\colon X_i\to Y_i$.  Let $\Phi_i$ be a range $R$ block code that implements $\varphi_i$, meaning that 
$$
\varphi_i(x)_t=\Phi_i(x_{t-R},\dots,x_t,\dots,x_{t+R}) 
$$
for all $t\in\Z$.  Let $\tilde{R}=\max\{\lceil(f-1)/2\rceil,R\}$ and let $\tilde{\Phi}_i$ be a range $\tilde{R}$ block code that also implements $\varphi_i$ (only making use of $(x_{t-R},\dots,x_{t+R})$ in the case that $(f-1)/2>R$).

We  prove $\per_{\leq p}(Y_1)\subseteq\per_{\leq p}(Y_2)$, the other case being analogous.  Let $y\in\per_{\leq p}(Y_1)$ and find $q\leq p$ such that $\sigma^qy=y$.  
Since $N>q\cdot a^{2\tilde{R}+1}$ and $\mathcal{L}_N(Y_2)=\mathcal{L}_N(Y_1)$, we can find $x\in X_2$ such that $\varphi_2(x)$ satisfies 
\begin{equation}\label{eq:word}
y_t=\varphi_2(x)_t \quad \text{ for all }0\leq t\leq q\cdot a^{2\tilde{R}+1}.  
\end{equation}
For each $0\leq k\leq a^{2\tilde{R}+1}$, set 
$$ 
w_k:=(x_{qk-\tilde{R}},\dots,x_{qk},\dots,x_{qk+\tilde{R}}).
$$ 
It follows from Equation~\eqref{eq:word} that $\tilde{\Phi}_2(w_k)=y_{qk}$ for all $0\leq k\leq a^{2\tilde{R}+1}$.   
By  definition of $\tilde{R}$, we have that $|w_k|\geq f$.  Since $|\mathcal{L}_{2\tilde{R}+1}(X_2)|\leq a^{2\tilde{R}+1}$, there exist $0\leq k_1<k_2\leq a^{2\tilde{R}+1}$ such that $w_{k_1}=w_{k_2}$.  Since $X_2$ can be defined using a set of forbidden words that all have length at most $f$, there is a periodic point $z\in X_2$ of period $q(k_2-k_1)$ such that 
$$ 
z_t=x_t \quad \text{ for all }qk_1-\tilde{R}\leq t\leq qk_2-\tilde{R}. 
$$ 
Note that this automatically also holds for $qk_2-\tilde{R}<t\leq qk_2+\tilde{R}$ in $z$ by periodicity, and in $x$ since $w_{k_1}=w_{k_2}$.  Therefore, it follows that 
\begin{equation}\label{eq:per}
\varphi_2(z)_t=\varphi_2(x)_t \quad \text{ for all }qk_1\leq t\leq qk_2. 
\end{equation}
Thus $\varphi_2(z)$ is periodic, with (not necessarily minimal) period $q(k_2-k_1)$, and because $\varphi_2(x)$ is a periodic word with period $q$, Equation~\eqref{eq:per} implies that $\varphi_2(z)$ is also periodic with period $q$.  
It follows that $y=\varphi_2(z)\in\per_q(Y_2)\subseteq\per_{\leq p}(Y_2)$.  Since $y\in\per_{\leq p}(Y_1)$ is arbitrary, it follows that $\per_{\leq p}(Y_1)\subseteq\per_{\leq p}(Y_2)$. 
\end{proof}

For a sofic shift $Y$, we are interested in constructing a measure supported on the periodic orbits of small period.  Of course $Y$ might not have any periodic points of very low period, and we use a lemma to provide an upper bound on the smallest period of any periodic point in $Y$.

\begin{lemma}\label{lem:nonempty}
Let $Y\in\mathcal{Y}_{a,f,R}$ for some $a, f, R\in\N$.  There exists $p\leq1+a^f$ such that $\per_p(Y)\neq\emptyset$.
\end{lemma}
\begin{proof}
Find $X\in\mathcal{X}_{a,f}$ and a factor map $\varphi\colon X\to Y$.  Since $|\mathcal{L}_f(X)|\leq a^f$, for any allowed word $x_1\cdots x_{a^f+f} \in \mathcal{L}_{a^f +f}(X)$ there exist $0\leq i<j\leq a^f$ such that 
$$
x_i\cdots x_{i+f-1}=x_j\cdots x_{j+f-1}. 
$$ 
Since $X$ can be defined using minimal forbidden words that all have length at most $f$, there is a periodic point $y  \in X$ of period $j-i$ such that $x_t=y_t$ for all $i\leq t\leq j+f$.  
It follows that $\varphi(y)\in Y$ is a periodic point of period dividing $j-i$.  In particular, $Y$ has a periodic point of period at most $1+a^f$.
\end{proof}

\begin{theorem}\label{th:char-periodic}
Let $X$ be a shift and $(X_n)_{n\in\N}$ be its SFT cover.  Let $Y$ be a factor of $X$ and $\varphi\colon X\to Y$ a factor map.  If there exist infinitely many $n\in \N$ such that $X_n=X_{n+\tau(n)}$ for the function 
$\tau(n)=2n+(1+n^n)\cdot n^{4n+1}$, 
then $Y$ has a characteristic measure.
\end{theorem}

We have not made any effort to optimize the growth rate $\tau(n)$, as this does not lead to stronger results in our setting.  
We note that this result  gives a direct proof that any sofic shift admits a characteristic measure.

\begin{proof}
Let $R_{\varphi}$ be a range for $\varphi$ and let $\Phi$ be a range $R_{\varphi}$ block code that implements it.  
The domain of $\Phi$ naturally extends to any shift $Z$ satisfying $\mathcal{L}_{2R_{\varphi}+1}(Z)=\mathcal{L}_{2R_{\varphi}+1}(X)$.  In particular, it extends to $X_n$ for all $n\geq2R_{\varphi}+1$.  Making a small abuse of notation, we use $\varphi$ to denote the map determined by $\Phi$ on each such $X_n$.  For $n\geq2R_{\varphi}+1$,  define $Y_n:=\varphi(X_n)$.  
It follows immediately from the definitions that $\mathcal{L}_k(Y_n)=\mathcal{L}_k(Y)$ for all $k\leq n-2R_{\varphi}$.

Let $\beta\in\Aut(Y)$ and let $R_{\beta}$ be a symmetric range for $\beta$.  Again, the domain of $\beta$ naturally extends to any shift $Z$ satisfying $\mathcal{L}_{2R_{\beta}+1}(Z)=\mathcal{L}_{2R_{\beta}+1}(Y)$.  In particular, the domain extends to $Y_n$ for all $n\geq2R_{\varphi}+2R_{\beta}+1$ (recall that we assume that the range of $\beta$ is symmetric).  For such $n$, 
since $\beta(Y)=Y$ it follows that $\mathcal{L}_k(\beta(Y_n))=\mathcal{L}_k(Y)$ for all $k\leq n-2R_{\beta}$.  Thus, the domain of the block code implementing $\beta^{-1}$ extends naturally to $\beta(Y_n)$ for all $n\geq4R_{\beta}+1$, provided $n$ is also at least $2R_{\varphi}+2R_{\beta}+1$ so that $Y_n$ is defined.  
Composing the block codes implementing $\beta$ and $\beta^{-1}$, the resulting block code of range
$2R_{\beta}$ implements the identity map on any shift $Z$ satisfying $\mathcal{L}_{4R_{\beta}+1}(Z)=\mathcal{L}_{4R_{\beta}+1}(Y)$.  In particular, these block codes implement the identity map on $Y_n$ for all $n\geq2R_{\beta}+\max\{2R_{\varphi},2R_{\beta}\}+1$.  It follows that, for such $n$, the map $\beta$ is a topological conjugacy between $Y_n$ and $\beta(Y_n)$.  Moreover, for all such $n$, 
\begin{equation}\label{eq:isom_lang}
\mathcal{L}_k(Y)=\mathcal{L}_k(Y_n)=\mathcal{L}_k(\beta(Y_n)) \quad \text{ for all }k\leq n-2R_{\beta}. 
\end{equation} 
Finally we observe that $Y_n$ is a factor of $X_n$ implemented by some factor map of range $R_{\varphi}$ and $\beta(Y_n)$ is a factor of $X_n$ implemented by some factor map of range $R_{\varphi}+R_{\beta}$.  Setting $a=\max\{|\mathcal{L}_1(X)|,|\mathcal{L}_1(Y)|\}$ we have, up to renaming the letters of the alphabets in $X$ and $Y$, that both $Y_n$ and $\beta(Y_n)$ are in $\mathcal{Y}_{a,n,R_{\varphi}+R_{\beta}}$.

Suppose $T,n\in\N$ are such that $X_n=X_{n+T}$ (and thus we also have that $Y_n=Y_{n+T}$).  Then applying Equation~\eqref{eq:isom_lang} to $Y_{n+T}$, it follows that  
\begin{equation}\label{eq:isom_lang2}
\mathcal{L}_k(Y)=\mathcal{L}_k(Y_n)=\mathcal{L}_k(\beta(Y_n)) \quad \text{ for all }k\leq n+T-2R_{\beta}. 
\end{equation} 
Note that we have written $Y_n$ instead of $Y_{n+T}$ and $\beta(Y_n)$ instead of $\beta(Y_{n+T})$.  Applying  Lemma~\ref{lem:language_periodic}, we have that $\per_{\leq p}(Y_n)=\per_{\leq p}(\beta(Y_n))$ for any $p$ such that $n+T-2R_{\beta}\geq p\cdot a^{\max\{n,2R_{\varphi}+2R_{\beta}+1\}}$.  In other words, this holds so long as 
$$ 
T\geq2R_{\beta}+p\cdot a^{\max\{n,2R_{\varphi}+2R_{\beta}+1\}}-n. 
$$ 
Moreover,  by Lemma~\ref{lem:nonempty}, we have that $\per_{\leq p}(Y_n)\neq\emptyset$ so long as $p\geq1+a^n$.  

Summarizing, if $n,T\in\N$ are such that $X_n=X_{n+T}$ and 
$$ 
T\geq2R_{\beta}+(1+a^n)\cdot a^{\max\{n,2R_{\varphi}+2R_{\beta}+1\}}-n 
$$ 
then there exists $p\leq1+a^n$ such that $\per_{\leq p}(Y_n)=\per_{\leq p}(\beta(Y_n))$ and these sets are nonempty.  
It follows that in this case we have 
\begin{equation}\label{eq:support}
\beta\left(\per_{\leq(1+a^n)}(Y_n)\right)=\per_{\leq(1+a^n)}(\beta(Y_n))
\end{equation} 
and so $\beta$ preserves the measure 
$$ 
\nu_{a,n,R_{\beta},R_{\varphi}}=\frac{1}{|\per_{\leq(1+a^n)}(Y_n)|}\sum_{z\in\per_{\leq(1+a^n)}(Y_n)}\delta_z. 
$$ 
In fact, this measure is preserved by any automorphism in $\Aut(Y)$ of range at most $R_{\beta}$.  But by the assumption on the SFT cover of $X$, there are infinitely many $n\in\N$ for which $Y_n=Y_{n+\tau(n)}$ and $\tau(n)\geq2R_{\beta}+(1+a^n)\cdot a^{\max\{n,2R_{\varphi}+2R_{\beta}+1\}}-n$ for all but finitely many such $n$.  Therefore, for all but finitely many such $n$, the measure $\nu_{n,n,n,n}$ is preserved by any automorphism of $Y$ whose range is at most $R_{\beta}$.  Since $\beta\in\Aut(Y)$, and hence $R_{\beta}$, is arbitrary, the measure $\nu_{n,n,n,n}$ is preserved by any automorphism of $Y$.  Any weak* limit of the sequence $\{\nu_{n,n,n,n}\}_{n=1}^{\infty}$ must be $\Aut(Y)$-invariant and since $Y=\bigcap_{n=1}^{\infty}Y_n$, the limiting measure is supported on $Y$.  Thus the factor $Y$ has a characteristic measure.
\end{proof}

It does not follow from the proof of Theorem~\ref{th:char-periodic} that the resulting characteristic measure is necessarily a measure of maximal entropy.  
Moreover, even though the set of measures of maximal entropy is  compact, convex and invariant under all automorphisms, as the automorphism group of a shift need not be amenable, it is not clear {\em a priori} that there exists a measure of maximal entropy invariant under any automorphism.
However, we show that with further assumptions on the growth of $\tau(n)$, we can guarantee this condition as well.  To help us do this, we need to better understand the measures of maximal entropy on a (not necessarily transitive) sofic shift $Y$ and this is carried out in the next section.

\subsection{Obtaining a characteristic measure that is of maximal entropy}

A shift $(X, \sigma)$ is {\em forward transitive} if for some $x\in X$, the forward orbit $\{\sigma^nx: n\in\N\}$ is dense in $X$.  
\begin{lemma}\label{lem:transitive_max}
Let $Y\in\mathcal{Y}_{a,f,R}$ for some $a,f,R\in\N$, $f> 2R+1$.  There exists a forward transitive sofic shift $Z\subseteq Y$ such that $Z\in\mathcal{Y}_{a,f,R}$ and $h_{\ttp}(Z)=h_{\ttp}(Y)$.  
Moreover, if $X\in\mathcal{X}_{a,f}$ and $\varphi\colon X\to Y$ is a range $R$ block code, then there exists $X^{\prime}\subseteq X$ such that $X^{\prime}\in\mathcal{X}_{a,f}$, $X^{\prime}$ is forward transitive and $Z=\varphi(X^{\prime})$.  In addition, if $\mu_Y$ is any ergodic measure of maximal entropy on $Y$, then $Z$ can be chosen such that $Z$ contains the support of $\mu_Y$.
\end{lemma}
\begin{proof}
Let $X\in\mathcal{X}_{a,f}$ and let $\varphi\colon X\to Y$ be a factor map that can be implemented by a range $R$ block code.  Since $Y$ is sofic, it has an ergodic measure of maximal entropy $\mu_Y$. 
By the pointwise ergodic theorem, for $\mu_Y$-almost every $y\in Y$ we have 
$$ 
\lim_{n\to\infty}\frac{1}{n}\sum_{k=0}^{n-1}\psi(\sigma^ky)=\int_Y\psi \,d\mu_Y 
$$ 
for all $\psi\in L^1(\mu_Y)$.   Fix one such $y\in Y$ and observe that $\mu_Y$ is supported on 
the $\omega$-limit set $\omega(y)\subseteq Y$.  In particular, $h_{\ttp}(\omega(y))\geq h_{\mu_Y}(\sigma)=h_{\ttp}(Y)$ and thus $h_{\ttp}(\omega(y))=h_{\ttp}(Y)$.

Choose $x\in X$ such that $y=\varphi(x)$.  Define 
$$ 
\mathcal{W}_x:=\{w\in\mathcal{L}_f(X)\colon\sigma^k(x)\in[w]\text{ for infinitely many }k\in\N\} 
$$ 
and $\mathcal{F}_x:=\{0,1,\dots,a-1\}^f\setminus\mathcal{W}_x$.  Let $X^{\prime}$ be the shift of finite type 
defined by forbidding the words in $\mathcal{F}_x$.
So we have that $X^{\prime}\in\mathcal{X}_{a,f}$ and $X' \subset X$.   Since they are only finitely many words of length $f$, ultimately, the sequence $x$ is a concatenation of words in $\mathcal{W}_x$. More precisely, there is some 
 $k_0\in\N$ such that $x_kx_{k+1}\cdots x_{k+f-1}\in\mathcal{L}_f(X^{\prime})$ for all $k\geq k_0$. Hence any word $w \in \mathcal{W}_x$ occurs in an infinite sequence in $X^{\prime}$ and a word $w\in\mathcal{L}(X)$ with $|w|\geq f$ is in the language 
$\mathcal{L}(X^{\prime})$ if and only if all of its subwords of length $f$ are in the language $\mathcal{L}_f(X^{\prime})$.

Since $X^{\prime}$ is defined by forbidden words of length $f$ and for any $u,v\in\mathcal{L}_f(X^{\prime})$ there exists $w\in\mathcal{L}(X^{\prime})$ such that $\sigma^k(x)\in[uwv]$ for some $k\geq k_0$, it follows that $X^{\prime}$ is forward transitive.  Setting  $Z:=\varphi(X^{\prime})\in\mathcal{Y}_{a,f,R}$, we have that $Z$ is  a forward transitive sofic shift in $Y$.  
Furthermore, since $X^{\prime}$ contains the $\omega$-limit set of $x$, it follows that $\omega(y)\subseteq Z$.  
Thus we have that $h_{\ttp}(Y)=h_{\ttp}(\omega(y))\leq h_{\ttp}(Z)$, which implies that $h_{\ttp}(Z)=h_{\ttp}(Y)$.
\end{proof}

Recall that any forward transitive sofic shift has a unique measure of maximal entropy~\cite{Weiss} 
and is {\em entropy minimal}~\cite{Coven-Smital}, meaning that all proper subshifts have strictly lower entropy.

\begin{lemma}\label{lem:finite_union}
Let $Y\in\mathcal{Y}_{a,f,R}$ for some $a,f,R\in\N$, and $f > 2R +1$.  There exist $k\in\N$ and sofic subshifts $Y_1,\dots,Y_k\subseteq Y$ such that 
	\begin{enumerate}
	\item For $i=1, \ldots, k$, each shift $Y_i\in\mathcal{Y}_{a,f,R}$; 
	\item For $i=1, \ldots, k$, each shift $Y_i$ is forward transitive; 
	\item For $i=1, \ldots, k$, we have $h_{\ttp}(Y_i)=h_{\ttp}(Y)$; 
	\item Every ergodic measure of maximal entropy supported on $Y$ is supported on $\bigcup_{i=1}^kY_i$; 
	\item For any $\alpha\in\Aut(Y)$, we have $\alpha(\bigcup_{i=1}^kY_i)=\bigcup_{i=1}^kY_i$.
	\end{enumerate} 
\end{lemma}
\begin{proof}
Let $X\in\mathcal{X}_{a,f}$ and $\varphi\colon X\to Y$ a factor map that can be implemented by a range $R$ block code.  By Lemma~\ref{lem:transitive_max}, for any ergodic measure of maximal entropy $\mu$ on $Y$, 
there exists a forward transitive shift $X_{\mu}\subseteq X$ such that $X_{\mu}\in\mathcal{X}_{a,f}$ and $\mu$ is supported on the forward transitive sofic shift $Y_{\mu}:=\varphi(X_{\mu})\in\mathcal{Y}_{a,f,R}$ 
which satisfies $h_{\ttp}(Y_{\mu})=h_{\ttp}(Y)$. 
Since a forward transitive sofic shift is entropy minimal, it follows that $Y_{\mu}$ is actually equal to the support of $\mu$.  Since $\mathcal{X}_{a,f}$ is finite, there are only 
finitely many distinct shifts $X_{\mu}$ and only finitely many distinct shifts $Y_{\mu}$.  
Enumerate the collection of all shifts that arise as $Y_{\mu}$, for some ergodic measure of maximal entropy $\mu$, as $Y_1,\dots,Y_k$.  
If $\alpha\in\Aut(X)$, then $\alpha$ permutes the ergodic measures of maximal entropy supported on $Y$, and therefore permutes the associated supports.  Therefore $\alpha$ permutes $Y_1,\dots,Y_k$ and so preserves $\bigcup_{i=1}^kY_i$.
\end{proof}

\begin{theorem}\label{th:char_sofic}
For any $a, f, R\in\N$ and $Y\in\mathcal{Y}_{a,f,R}$, the subshift $Y$ has a characteristic measure of maximal entropy that is a weak* limit of the sequence of measures $\{\nu_n\}_{n=1}^{\infty}$ where 
$$ 
\nu_n:=\frac{1}{|\per_{\leq n}(Y)|}\sum_{z\in \per_{\leq n}(Y)}\delta_z. 
$$ 
Moreover, if $\mathcal{S}\subseteq\N$ is infinite, there is a characteristic measure of maximal entropy that is  obtained as the weak* limit along a subsequence of elements of $\mathcal{S}$.
\end{theorem}
In particular, this result shows that every accumulation point in the weak* topology of the set of measures $\{\nu_n\}_{n=1}^{\infty}$ is characteristic and is the measure of maximal entropy.

\begin{proof}
Let $Y\in\mathcal{Y}_{a,f,R}$ and let $Y_1,\dots,Y_k\subseteq Y$ be as in Lemma~\ref{lem:finite_union}.  Suppose $\mathcal{S}\subseteq\N$ is given.  For each $i=1,\dots,k$ let $\mu_i$ be the unique measure of maximal entropy supported on $Y_i$.  We first show any measure of the form 
$$ 
\mu:=\sum_{i=1}^k\tilde{c}_i\mu_i 
$$ 
is a characteristic measure on $Y$, provided $\tilde{c}_i\in[0,1]$ for all $i=1, \dots, k$, $\sum c_i=1$, and whenever $Y_i$ and $Y_j$ are topologically conjugate we have $\tilde{c}_i=\tilde{c}_j$.

Any $\alpha\in\Aut(Y)$ permutes $Y_1,\dots,Y_k$ and so permutes the measures $\mu_1,\dots,\mu_k$, 
as each $\mu_i$ is the unique measure of maximal entropy on $Y_i$ for all $=1, \dots, ki$.  Note that if this permutation sends $\mu_i$ to $\mu_j$ then $Y_i$ is topologically conjugate to $Y_j$ and so $\tilde{c}_i=\tilde{c}_j$.  Therefore $\alpha_*\mu=\mu$.  Since $\alpha\in\Aut(Y)$ is arbitrary, $\mu$ is an $\Aut(Y)$-characteristic measure.

Next, find a sequence $\{n_t\}_{t=1}^{\infty}$ of elements of $\mathcal{S}$ along which 
$$ 
c_i:=\lim_{t\to\infty}\frac{\left|\per_{\leq n_t}(Y_i)\right|}{\left|\bigcup_{j=1}^k\per_{\leq n_t}(Y_j)\right|}
$$ 
exists for all $i=1,\dots,k$.  Note that $c_i\in[0,1]$ for all $i=1, \dots, k$.  Entropy minimality of $Y_i$ implies that for any $i<j$, we have that $h_{\ttp}(Y_i\cap Y_j)<h_{\ttp}(Y_i)$.  Thus it follows that 
$$ 
\lim_{n\to\infty}\frac{|\per_{\leq n}(Y_i\cap Y_j)|}{|\per_{\leq n}(Y_i)|}=0. 
$$ 
Defining $P_n(Y_i):=\per_{\leq n}(Y_i)\setminus\left(\bigcup_{j\neq i}\per_{\leq n}(Y_j)\right)$,  we have that 
\begin{equation}\label{eq:switch}
\lim_{n\to\infty}\frac{|P_n(Y_i)|}{|\per_{\leq n}(Y_i)|}=1 
\end{equation} 
and so 
\begin{equation}\label{eq:switch2}
c_i=\lim_{t\to\infty}\frac{\left|P_{n_t}(Y_i)\right|}{\left|\bigcup_{j=1}^kP_{n_t}(Y_j)\right|}. 
\end{equation} 
Finally note that if $Y_i$ is topologically conjugate to $Y_j$ then $|\per_{\leq n_t}(Y_i)|=|\per_{\leq n_t}(Y_j)|$ for all $t$, so $c_i=c_j$.  It follows that 
$$ 
\mu_Y:=\sum_{i=1}^kc_i\mu_i 
$$ 
is an $\Aut(Y)$-characteristic measure. 

We next show that $\mu_Y$ is the weak* limit of the sequence $\{\nu_{n_t}\}_{t=1}^{\infty}$. 
Combining the results of Bowen~\cite[Theorem~34]{bowen} and~\cite[Corollary~6.7]{bowen2}, the periodic points are equidistributed with respect to the measure of maximal entropy in any forward transitive subshift.  Thus 
since $Y_i$ is a forward transitive sofic shift for each $i=1,\dots,k$, 
we have that 
$$ 
\frac{1}{|\per_{\leq n_t}(Y_i)|}\sum_{z\in \per_{\leq n_t}(Y_i)}\delta_z\xrightarrow[t\to\infty]{} \mu_i.
$$ 
Combining this with Equation~\eqref{eq:switch} we have that  
$$ 
\frac{1}{|P_{n_t}(Y_i)|}\sum_{z\in P_{n_t}(Y_i)}\delta_z\xrightarrow[t\to\infty]{} \mu_i.
$$ 
Therefore, $\mu_Y$ is the weak* limit 
$$ 
\lim_{t\to\infty}\sum_{i=1}^k\frac{\left|P_{n_t}(Y_i)\right|}{\left|\bigcup_{j=1}^kP_{n_t}(Y_j)\right|}\cdot\frac{1}{|P_{n_t}(Y_i)|}\sum_{z\in P_{n_t}(Y_i)}\delta_z=\lim_{t\to\infty}\frac{1}{\left|\bigcup_{j=1}^kP_{n_t}(Y_j)\right|}\sum_{z\in\bigcup_{j=1}^kP_{n_t}(Y_j)}\delta_z. 
$$ 
Therefore, again using from~\eqref{eq:switch}, it follows that $\mu_Y$ is the weak-* limit of the sequence 
$$ 
\frac{1}{\left|\per_{\leq n_t}\left(\bigcup_{i=1}^kY_i\right)\right|}\sum_{z\in \per_{\leq n_t}\left(\bigcup_{i=1}^kY_i\right)}\delta_z. 
$$ 

Finally, let $X\in\mathcal{X}_{a,f}$ and let $\varphi\colon X\to Y$ be a factor map that can be implemented by a range $R$ block code.  For each periodic point $y\in Y$, 
there is a periodic point $x\in X$ such that 
$y=\varphi(x)$.  Let $X^{\prime}\subseteq X$ be the shift of finite type obtained by forbidding all words of length $f$ that do not appear in the point $x$.  Note that $X^{\prime}\in\mathcal{X}_{a,f}$, $X^{\prime}$ is forward transitive, and $x\in X^{\prime}$.  Therefore $y\in\varphi(X^{\prime})$.  
It follows that every periodic point in $Y$ lies in $\varphi(X^{\prime})$ for some $X^{\prime}\subseteq X$ with $X^{\prime}\in\mathcal{X}_{a,f}$.  Since there are only finitely many such subshifts and all subshifts with entropy $h_{\ttp}(Y)$ already appear in the enumeration $Y_1,\dots,Y_k$, it follows that all periodic points in $Y\setminus\left(\bigcup_{i=1}^kY_i\right)$ are in the union of a finite number of shifts of the form $\varphi(Z)$, where $Z\in\mathcal{X}_{a,f}$ and $h_{\ttp}(\varphi(Z))<h_{\ttp}(Y)$.  Therefore 
$$ 
\lim_{n\to\infty}\frac{\left|\per_{\leq n}\left(\bigcup_{i=1}^kY_i\right)\right|}{|\per_{\leq n}(Y)|}=1 
$$ 
and so $\mu_Y$ is the weak* limit of the sequence 
\begin{equation*}
\frac{1}{|\per_{\leq n_t}(Y)|}\sum_{z\in \per_{\leq n_t}(Y)}\delta_z. \quad\qedhere
\end{equation*}
\end{proof}

\begin{lemma}\label{lem:consistent}
There exists an infinite subset $\mathcal{S}\subseteq\N$ such that for all $a,f,R\in\N$ and all $Y\in\mathcal{Y}_{a,f,R}$, the weak* limit 
\begin{equation}
\label{eq:limit-exists}
\mu_Y:=\lim_{s\in\mathcal{S}, s\to\infty}\frac{1}{|\per_{\leq s}(Y)|}\sum_{z\in \per_{\leq s}(Y)}\delta_z
\end{equation}
exists and is a characteristic measure of maximal entropy on $Y$. 
\end{lemma}
\begin{proof}
Note that 
$$ 
\bigcup_{a=1}^{\infty}\bigcup_{f=1}^{\infty}\bigcup_{R=1}^{\infty}\mathcal{Y}_{a,f,R}=\bigcup_{n=1}^{\infty}\mathcal{Y}_{n,n,n} 
$$ 
and so  it suffices to ensure sure that the limit in~\eqref{eq:limit-exists} exists for all $Y\in\bigcup_{n=1}^{\infty}\mathcal{Y}_{n,n,n}$.  We construct $\mathcal{S}$ iteratively using a diagonalization argument.  Enumerate the elements (of the finite set) $\mathcal{Y}_{1,1,1}$ as $Y_1,\dots,Y_r$.  Begin by setting $\mathcal{S}_0:=\N$.  By Theorem~\ref{th:char_sofic}, there exists a sequence $\{n_t\}_{t=1}^{\infty}$ along which 
$$
\lim_{t\to\infty}\frac{1}{|\per_{\leq n_t}(Y_1)|}\sum_{z\in \per_{\leq n_t}(Y_1)}\delta_z
$$
exists and is a characteristic measure of maximal entropy on $Y_1$.  Set $\mathcal{S}_1:=\{1\}\cup\{n_t\colon t\in\N\}$.  Suppose we have constructed nested infinite sets 
$$ 
\mathcal{S}_0\supseteq\mathcal{S}_1\supseteq\cdots\supseteq\mathcal{S}_u 
$$ 
such that $\mathcal{S}_u$ contains the $v+1$ smallest elements of $\mathcal{S}_v$ for all $v<u$ and such that 
$$ 
\lim_{s\in\mathcal{S}_v, s\to\infty}\frac{1}{|\per_{\leq s}(Y_v)|}\sum_{z\in \per_{\leq s}(Y_v)}\delta_z
$$ 
exists and is a characteristic measure of maximal entropy on $Y_v$ for all $0<v\leq u$.  If $u=r$ this part of the construction ends and we move onto the next step.  Otherwise, use Theorem~\ref{th:char_sofic} with the infinite set $\mathcal{S}_u$ to find a subsequence $\{n_t\}_{t=1}^{\infty}$ of elements of $\mathcal{S}_u$ for which 
$$ 
\lim_{t\to\infty}\frac{1}{|\per_{\leq n_t}(Y_{u+1})|}\sum_{z\in \per_{\leq n_t}(Y_{u+1})}\delta_z
$$ 
exists and is a characteristic measure on $Y_{u+1}$.  Define the set $\mathcal{S}_{u+1}$ to be the union of $\{n_t\colon t\in\N\}$ with the $u$ smallest elements of $\mathcal{S}_u$.  This part of the construction terminates when we construct $\mathcal{S}_r$.  For convenience in the next step, we define $\mathcal{S}^{(1)}:=\mathcal{S}_r$. 

Suppose we have defined a nested sequence of infinite sets 
$$ 
\mathcal{S}^{(1)}\supseteq\mathcal{S}^{(2)}\supseteq\cdots\supseteq\mathcal{S}^{(j)} 
$$ 
such that for all $Y\in\mathcal{Y}_{j,j,j}$ the weak* limit 
$$ 
\lim_{s\in\mathcal{S}^{(j)},s\to\infty}\frac{1}{|\per_{\leq s}(Y)|}\sum_{z\in \per_{\leq s}(Y)}\delta_z
$$
exists and is a characteristic measure of maximal entropy on $Y$.  Suppose further that $\mathcal{S}^{(j)}$ contains the $i$ smallest elements of $\mathcal{S}^{(i)}$ for all $i<j$.  Proceed as in the construction of $\mathcal{S}^{(1)}$, but using $\mathcal{S}^{(j)}$ in place of $\mathcal{S}_0$ at the start, to build an infinite set $\mathcal{S}^{(j+1)}\subseteq\mathcal{S}^{(j)}$ that contains the $j$ smallest elements of $\mathcal{S}^{(j)}$ and for which 
$$ 
\lim_{s\in\mathcal{S}^{(j+1)},s\to\infty}\frac{1}{|\per_{\leq s}(Y)|}\sum_{z\in \per_{\leq s}(Y)}\delta_z
$$
exists and is a characteristic measure of maximal entropy on $Y$ for all $Y\in\mathcal{Y}_{j+1,j+1,j+1}$.

Continuing inductively, we construct an infinite sequence of nested, infinite sets 
$$ 
\mathcal{S}^{(1)}\supseteq\mathcal{S}^{(2)}\supseteq\cdots\supseteq\mathcal{S}^{(j)}\supseteq\cdots 
$$ 
such that for any $i<j$ the set $\mathcal{S}^{(j)}$ contains the $i$ smallest elements of $\mathcal{S}^{(i)}$ and such that for any $Y\in\mathcal{Y}_{j,j,j}$ the weak* limit 
$$ 
\lim_{s\in\mathcal{S}^{(j)},s\to\infty}\frac{1}{|\per_{\leq s}(Y)|}\sum_{z\in \per_{\leq s}(Y)}\delta_z
$$
exists and is a characteristic measure of maximal entropy on $Y$.  We define the (nonempty by construction) set $\mathcal{S}:=\bigcap_{j=1}^{\infty}\mathcal{S}^{(j)}$.  
\end{proof}

We use this to motivate the relevant class of shifts and approximations.  
\begin{definition}
\label{def:const}
Let $\mathcal{S}\subseteq\N$ be the set constructed in Lemma~\ref{lem:consistent}.  For $a,f,R\in\N$ and $Y\in\mathcal{Y}_{a,f,R}$, let 
$$
\mu_Y:=\lim_{s\in\mathcal{S},s\to\infty}\frac{1}{|\per_{\leq s}(Y)|}\sum_{z\in \per_{\leq s}(Y)}\delta_z. 
$$
For $k,m\in\N$, define $A(a,f,R,k,m)\in\N$ to be the least element of $\mathcal{S}$ such that for all $Y\in\mathcal{Y}_{a,f,R}$ and all $w\in\bigcup_{i=1}^m\mathcal{L}_i(Y)$ we have 
$$ 
\left|\mu_Y([w])-\frac{1}{|\per_{\leq s}(Y)|}\sum_{z\in \per_{\leq s}(Y)}\delta_z([w])\right|<\frac{1}{k} 
$$ 
for all $s\in\mathcal{S}$ with $s\geq A(a,f,R,k,m)$.

For $n\in\N$, let $m=m(n)$ be 
$$ 
m:=\inf\{\mathcal{S}\setminus\{1,2,\dots,n\}\}. 
$$ 
Define $\omega\colon\N\to\N$ by 
\begin{equation} 
\label{def:omega}
\omega(n):=-n+2m+A(m,m,m,m,m)\cdot m^{2m+1}
\end{equation} 
\end{definition}

We are now ready to formulate a version of Theorem~\ref{th:char-periodic} that 
relies on a stronger assumption on  how well-approximable the shift is to obtain the stronger conclusion that the resulting characteristic measure is a measure of maximal entropy.

\begin{proposition}
\label{prop:well-approx-factors}
Let $X$ be a shift and $(X_n)_{n\in\N}$ its SFT cover.  
If there are infinitely many $n\in\N$ such that $X_n=X_{n+\omega(n)}$ for the function $\omega(n)$ defined in~\eqref{def:omega}, then  every subshift factor of $X$ has a characteristic measure which is a measure of maximal entropy.
\end{proposition}

We note that Example~\ref{ex:nonempty} shows that the class of shifts verifying this condition is larger than that of shifts of finite type. 

\begin{proof}
Let 
$$ 
\tau(n):=A(n,n,n,n,n)\cdot n^{2n+1}+n 
$$ 
and note that, by definition of $\omega$, for any $n\in\N$ such that $X_n=X_{n+\omega(n)}$, there exists $m\geq n$ with $m\in\mathcal{S}$ and such that $X_m=X_{m+\tau(m)}$.  Therefore our assumption that there are infinitely many $n$ such that $X_n=X_{n+\omega(n)}$ implies that there exist infinitely many $m\in\mathcal{S}$ such that $X_m=X_{m+\tau(m)}$.  To avoid unnecessary confusion, we make a small abuse of notation and simply state that 
\begin{equation}\label{eq:secret_assumption}
\text{there are infinitely many $n\in\mathcal{S}$ such that }X_n=X_{n+\tau(n)}.
\end{equation}

Let $Y$ be a subshift factor of $X$  and  let $\varphi\colon X\to Y$ be a factor map with range $R_{\varphi}$.  
  Setting $a:=\max\{|\mathcal{L}_1(X)|,|\mathcal{L}_1(Y)|\}$, without loss of generality (renaming the letters of the alphabets if necessary) we can assume that $X,Y\subseteq\{0,1,\dots,a-1\}^{\Z}$.  
Since $X_n$ is defined by forbidding only words of length $n$, we have that 
$X_n\in\mathcal{X}_{a,n}$ for all $n\in\N$.  

We begin as in the proof of Theorem~\ref{th:char-periodic}.  For each $n\geq2R_{\varphi}+1$, 
define $Y_n:=\varphi(X_n)$ and observe that $\mathcal{L}_k(Y_n)=\mathcal{L}_k(Y)$ for all $k\leq n-2R_{\varphi}$.  
For all such $n$, we have that $Y_n\in\mathcal{Y}_{a,n,R_{\varphi}}$.  
Let $\beta\in\Aut(Y)$ and let $R_{\beta}$ be a symmetric range for $\beta$.  For any $n\geq2R_{\varphi}+2R_{\beta}+1$ and any $k\leq n-2R_{\beta}$, 
we have that $\mathcal{L}_k(\beta(Y_n))=\mathcal{L}_k(Y)$.  Provided that $n\geq2R_{\beta}+\max\{2R_{\varphi},2R_{\beta}\}+1$, we have that $\beta$ is a topological conjugacy between $Y_n$ and $\beta(Y_n)$ and that 
$$ 
\mathcal{L}_k(Y)=\mathcal{L}_k(Y_n)=\mathcal{L}_k(\beta(Y_n)) \quad \text{ for all }k\leq n-2R_{\beta}. 
$$ 
It follows immediately that $\beta(Y_n)\in\mathcal{Y}_{a,n,R_{\varphi}+R_{\beta}}$ for all such $n$ and that $Y_n\in\mathcal{Y}_{a,n,R_{\varphi}}\subseteq\mathcal{Y}_{a,n,R_{\varphi}+R_{\beta}}$.  For any $n$ such that $X_n=X_{n+T}$ 
(meaning we also have that $Y_n=Y_{n+T}$), we have that 
$$
\mathcal{L}_k(Y)=\mathcal{L}_k(Y_n)=\mathcal{L}_k(\beta(Y_n)) \quad \text{ for all }k\leq n+T-2R_{\beta}. 
$$
For such $n$, using the parameters $a$, $f:=n$, $R:=R_{\varphi}+R_{\beta}$, and $N:=n+T-2R_{\beta}$ in Lemma~\ref{lem:language_periodic}, it follows that 
$\per_{\leq p}(Y_n)=\per_{\leq p}(\beta(Y_n))$ so long as $T>p\cdot a^{\max\{n,2R_{\varphi}+2R_{\beta}+1\}}+2R_{\beta}-n$.  In particular, when 
$$T>A(n,n,n,n,n)\cdot a^{\max\{n,2R_{\varphi}+2R_{\beta}+1\}}+2R_{\beta}-n,$$ 
it follows that 
\begin{equation}\label{eq:same}
\per_{\leq A(n,n,n,n,n)}(Y_n)=\per_{\leq A(n,n,n,n,n)}(\beta(Y_n)). 
\end{equation}
For fixed $\varphi$ and $\beta$, by Equation~\eqref{eq:secret_assumption}, there are infinitely many $n\in\mathcal{S}$ for which $X_n=X_{n+\tau(n)}$ where 
$$
\tau(n)=A(n,n,n,n,n)\cdot n^{2n+1}+n>A(n,n,n,n,n)\cdot a^{\max\{n,2R_{\varphi}+2R_{\beta}+1\}}+2R_{\beta}-n
$$ 
(since $n>\max\{R_{\varphi}+R_{\beta},a\}$ for all but finitely many $n$).  For any such $n$, recall that we have $Y_n,\beta(Y_n)\in\mathcal{Y}_{a,n,R_{\varphi}+R_{\beta}}\subseteq\mathcal{Y}_{n,n,n}$.  
Therefore, by definition of $A(n,n,n,n,n)$ and Theorem~\ref{th:char_sofic}, for any $w\in\mathcal{L}_n(Y_n)$ we have 
(for convenience, we write $A(n) = A(n,n,n,n,n)$ in this calculation)
\begin{align*} 
|\mu_{Y_n}([w])  -\mu_{\beta(Y_n)}([w])| &  \leq   \Big\vert\mu_{Y_n}([w])-\frac{1}{|\per_{\leq A(n)}(Y_n)|}\sum_{z\in \per_{\leq A(n)}(Y_n)}\delta_z([w])\Big\vert \\ 
 + & \Big\vert\mu_{\beta(Y_n)}([w])-\frac{1}{|\per_{\leq A(n)}(Y_n)|}\sum_{z\in \per_{\leq A(n)}(Y_n)}\delta_z([w])\Big\vert \\
 = & \Big\vert\mu_{Y_n}([w])-\frac{1}{|\per_{\leq A(n)}(Y_n)|}\sum_{z\in \per_{\leq A(n)}(Y_n)}\delta_z([w])\Big\vert \\ 
 +  & \Big\vert\mu_{\beta(Y_n)}([w])-\frac{1}{|\per_{\leq A(n)}(\beta(Y_n))|}\sum_{z\in \per_{\leq A(n)}(\beta(Y_n))}\delta_z([w])\Big\vert, 
\end{align*} 
where we use Equation~\eqref{eq:same} to make the change from $Y_n$ to $\beta(Y_n)$.  By definition of $A(n,n,n,n,n)$ and the fact that $Y_n,\beta(Y_n)\in\mathcal{Y}_{n,n,n}$, it follows 
that for any $w\in\mathcal{L}_n(Y_n)=\mathcal{L}_n(\beta(Y_n))$ we have  that 
\begin{equation}\label{eq:dist}
|\mu_{Y_n}([w])-\mu_{\beta(Y_n)}([w])|\leq\frac{2}{n}. 
\end{equation} 
By Theorem~\ref{th:char_sofic}, the measure $\mu_{Y_n}$ is the average of all ergodic measures of maximal entropy on $Y_n$ and $\mu_{\beta(Y_n)}$ is the average of all ergodic measures of maximal entropy on $\beta(Y_n)$.  
Since $\beta$ is a topological conjugacy between $Y_n$ and $\beta(Y_n)$, it follows that $\beta_*\mu_{Y_n}=\mu_{\beta(Y_n)}$.  Combining this with Equation~\eqref{eq:dist}, if $\mu$ is any weak* limit of the sequence $\{\mu_{Y_n}\}_{n=1}^{\infty}$, 
 then $\beta_*\mu=\mu$.  Since $\beta\in\Aut(Y)$ is arbitrary, any such weak* limit is an $\Aut(Y)$-characteristic measure supported on $Y=\bigcap_{n=1}^{\infty}Y_n$.  Finally, note that $h_{\mu_{Y_n}}(\sigma)=h_{\ttp}(Y_n)\geq h_{\ttp}(Y)$ for all $n\in\N$ and  so any weak* limit point is also a measure of maximal entropy on $Y$.
\end{proof}

\subsection{Not every shift is a factor of a well-approximable language stable shift}
\label{sec:not-all-wa}

Suppose $(a_n)_{n\in\N}$ is a non-decreasing sequence of positive integers such that  $\lim_na_n=\infty$.  Define $\mathcal{Z}(a_n)$ to be the set of all shifts $(X,\sigma)$ such that $X$ has no minimal forbidden words of length between $n$ and $n+a_n$ for infinitely many $n\in\N$.
Taking all choices of such sequences $(a_n)_{n\in\N}$, we obtain all of the language stable shifts, and we show that there is some shift that is not a factor of $\mathcal{Z}(a_n)$ for some particular choice of  sequence.

\begin{proposition}\label{prop:nonFactor}
There exists a non-decreasing sequence of positive integers $(a_n)_{n\in\N}$  and a subshift $B\subseteq\{0,1\}^{\Z}$ such that $B$ is not a factor of any element of $\mathcal{Z}(a_n)$.
\end{proposition}
\begin{proof}
We construct the shift $B$ recursively, constructing a nested sequence of subshifts and then defining $B$ to be their intersection.

Let $\mathcal{A}_0:=\{0,1\}$.  For each $t\in\N$, let $B_{0,t}$ denote the subshift of all possible bi-infinite concatenations of the words $0(1^{2t})$ and $0(1^{2t+1})$.  Note that 
\begin{equation}\label{eq:lang}
\mathcal{L}_{4t+3}(B_{0,t})\cap\mathcal{L}_{4t+3}(B_{0,t^{\prime}})=\emptyset  \quad\text{ for all }t<t^{\prime}, 
\end{equation}
since every element of $\mathcal{L}_{4t+3}$ contains either $01^{2t}0$ or $01^{2t+1}0$ as a subword whereas no element of $\mathcal{L}_{4t+3}(B_{0,t^{\prime}})$ does.  Observe that $\mathcal{L}_1(B_{0,t})=\{0,1\}$ for all $t\in\N$.  Find some $t_0>|\mathcal{Y}_{1,1,1}|$ such that 
$$ 
\mathcal{L}_{4|\mathcal{Y}_{1,1,1}|+3}(B_{0,t_0})\neq\mathcal{L}_{4|\mathcal{Y}_{1,1,1}|+3}(Y) \quad \text{ for all }Y\in\mathcal{Y}_{1,1,1}
$$ 
which is possible by~\eqref{eq:lang}.  Let $u_0\in\mathcal{L}(B_{0,t_0})$ be such that $u_0$ is a concatenation of the words $01^{2t_0}$ and $01^{2t_0+1}$ and every element of $\mathcal{L}_{4|\mathcal{Y}_{1,1,1}|+3}(B_{0,t_0})$ occurs as a subword of $u_0$.  We define two words $w_{0,0}$ and $w_{1,0}$ which play the role of the alphabet in the next stage of our construction, meaning that every element of the shift we construct at  this stage is a bi-infinite concatenation of these words.  
Set $w_{0,0}:=u_0(01^{2t_0})^{(2t_0+2)|u_0|(4|\mathcal{Y}_{1,1,1}|+3)}$ and 
 $w_{1,0}=u_0(01^{2t_0+1})^{(2t_0+1)|u_0|(4|\mathcal{Y}_{1,1,1}|+3)}$.  The words $w_{0,0}$ and $w_{1,0}$ have the following properties: 
	\begin{enumerate}
	\item $|w_{0,0}|=|w_{1,0}|$; 
	\item every element of $\mathcal{L}_{4|\mathcal{Y}_{1,1,1}|+3}(B_{0,t_0})$ occurs as a subword of $w_{0,0}$ and $w_{1,0}$; 
	\item any $\mathbb{Z}$-coloring that can be written as a bi-infinite concatenation of $w_{0,0}$ and $w_{1,0}$ can be written uniquely in this way (this is done by identifying where the word $u_0$ occurs); 
	\item for any $a,b\in\{0,1\}$, we have $w_{a,0}w_{b,0}\in\mathcal{L}(B_{0,t_0})$. 
	\end{enumerate} 
	
Inductively, suppose we have constructed words $w_{0,r}$ and $w_{1,r}$ which have the same lengths 
and the property that any $\mathbb{Z}$-coloring that can be written as a bi-infinite concatenation of $w_{0,r}$ and $w_{1,r}$ can be written uniquely in this way.  Further suppose we have constructed an integer $t_r$ and a subshift $B_{r,t_r}$ such that 
$$ 
\mathcal{L}_{4|\mathcal{Y}_{r,r,r}|+3}(B_{r,t_r})\neq\mathcal{L}_{4|\mathcal{Y}_{1,1,1}|+3}(Y) \quad \text{ for all }Y\in\mathcal{Y}_{r,r,r} 
$$ 
and that every element of $\mathcal{L}_{4|\mathcal{Y}_{r,r,r}|+3}(B_{r,t_r})$ occurs as a subword of both $w_{0,r}$ and $w_{1,r}$.  Also suppose $|w_{0,r}|, |w_{1,r}|>4|\mathcal{Y}_{r,r,r}|+3$.  Finally suppose that for any $a,b\in\{0,1\}$ we have $w_{a,r}w_{b,r}\in\mathcal{L}(B_{r,t_r})$.  We now mimic the construction of $B_{0,t_0}$ with 
$w_{0,r}$ and $w_{1,r}$  playing 
the role of an alphabet.  For each $t\in\N$, let $B_{r+1,t}$ be the subshift of all possible bi-infinite concatenations of the words $w_{0,r}(w_{1,r}^{2t})$ and $w_{0,r}(w_{1,r}^{2t+1})$.  Note that $\mathcal{L}_{4|\mathcal{Y}_{r,r,r}|+3}(B_{r+1,t})=\mathcal{L}_{4|\mathcal{Y}_{r,r,r}|+3}(B_{r,t_r})$ because every element of $\mathcal{L}_{4|\mathcal{Y}_{r,r,r}|+3}(B_{r,t_r})$ occurs as a subword of both $w_{0,r}$ and $w_{1,r}$.  Moreover, we have  that 
\begin{equation}\label{eq:longer_lang}
\mathcal{L}_{(4t+4)|w_{0,r}|}(B_{r+1,t})\cap\mathcal{L}_{(4t+4)|w_{0,r}|}(B_{r+1,t})=\emptyset \quad \text{ for all }t<t^{\prime} 
\end{equation}
because every element of $\mathcal{L}_{(4t+4)|w_{0,r}|}(B_{r+1,t})$ contains either the word $w_{0,r}w_{1,r}^{2t}w_{0,r}$ or the word $w_{0,r}w_{1,r}^{2t+1}w_{0,r}$ as a subword, 
whereas no element of $\mathcal{L}_{(4t+4)|w_{0,r}|}(B_{r+1,t^{\prime}})$ does (note the change from $4t+3$, used in the first stage of the construction, to $4t+4$ used now, which accounts for the fact that the words $w_{0,r}$ and $w_{1,r}$ can now be partially overlapped at the ends of words).  Choose  some $t_{r+1}>\max\{|\mathcal{Y}_{r+1,r+1,r+1}|,t_r\}$ such that 
$$ 
\mathcal{L}_{(4t+4)|w_{0,r}|}(B_{r+1,t})\neq\mathcal{L}_{(4t+4)|w_{0,r}|}(Y) \quad \text{ for any }Y\in\mathcal{Y}_{r+1,r+1,r+1},  
$$ 
which is possible by~\eqref{eq:longer_lang} and the fact that $\mathcal{Y}_{r+1,r+1,r+1}$ is finite.  Let $u_{r+1}\in\mathcal{L}(B_{r+1,t_{r+1}})$ be such that $u_{r+1}$ is a concatenation of the words $w_{0,r}(w_{1,r}^{2t})$ and $w_{0,r}(w_{1,r}^{2t+1})$ and every element of $\mathcal{L}_{(4|\mathcal{Y}_{r+1,r+1,r+1}|+4)|w_{0,r}|}(B_{r+1,t_{r+1}})$ occurs as a subword of $u_{r+1}$.  Define words 
	\begin{eqnarray*} 
	w_{0,r+1}&:=&u_{r+1}(w_{0,r}w_{1,r}^{2t_{r+1}})^{(2t_{r+1}+2)|u_{r+1}|(4|\mathcal{Y}_{r+1,r+1,r+1}|+4)} \\ 
	w_{1,r+1}&:=&u_{r+1}(w_{0,r}w_{1,r}^{1+2t_{r+1}})^{(2t_{r+1}+1)|u_{r+1}|(4|\mathcal{Y}_{r+1,r+1,r+1}|+4)} 
	\end{eqnarray*} 
These words have the following features: 
	\begin{enumerate}
	\item $|w_{0,r+1}|=|w_{1,r+1}|$; 
	\item every element of $\mathcal{L}_{(4|\mathcal{Y}_{r+1,r+1,r+1}|+4)|w_{0,r}|}(B_{r+1,t_{r+1}})$ occurs as a subword of $w_{0,r+1}$ and $w_{1,r+1}$; 
	\item any $\mathbb{Z}$-coloring that can be written as a bi-infinite concatenation of $w_{0,r+1}$ and $w_{1,r+1}$ can be written uniquely in this way, by identifying where the word $u_{r+1}$ occurs; 
	\item for any $a,b\in\{0,1\}$, we have $w_{a,r+1}w_{b,r+1}\in\mathcal{L}(B_{r+1,t_{r+1}})$. 
	\end{enumerate} 
Therefore this procedure inductively defines the shift $B_{r,t_r}$ for all $r\in\N$.  By construction, for any $r<r^{\prime}$ we have tgat 
$$ 
\mathcal{L}_{(4|\mathcal{Y}_{r,r,r}|+4)|w_{0,r}|}(B_{r,t_r})=\mathcal{L}_{(4|\mathcal{Y}_{r,r,r}|+4)|w_{0,r}|}(B_{r^{\prime},t_{r^{\prime}}}). 
$$ 
Furthermore, $B_{r+1,t_{r+1}}\subseteq B_{r,t_r}$ for all $r\in\N$.  Define the shift 
$$ 
B:=\bigcap_{r=1}^{\infty}B_{r,t_r}. 
$$ 
Then for any $r\in\N$, we have that 
\begin{equation}\label{eq:contra}
\mathcal{L}_{(4|\mathcal{Y}_{r,r,r}|+4)|w_{0,r}|}(B)\neq\mathcal{L}_{(4|\mathcal{Y}_{r,r,r}|+4)|w_{0,r}|}(Y) \quad \text{ for any }Y\in\mathcal{Y}_{r,r,r}. 
\end{equation}

We next show that $B$ is not a factor of any element of $\mathcal{Z}(a_n)$, where the
the sequence $(a_n)$ is defined by 
$$ 
a_n:=2\cdot(4|\mathcal{Y}_{n,n,n}|+4)|w_{0,n}|. 
$$ 
For contradiction, suppose $B$ is a factor of some shift $X\in\mathcal{Z}(a_n)$.  Let $\varphi\colon X\to B$ be a factor map and let $R$ be its (symmetric) range.  Find some $n>\max\{R,|\mathcal{L}_1(X)|\}$ such that $X$ has no minimal forbidden words of lengths between $n$ and $n+2(4|\mathcal{Y}_{n,n,n}|+4)|w_{0,n}|$.  If $X_n$ is the $n^{th}$ term in the SFT cover of $X$, then we have that  
$$ 
\mathcal{L}_{(4|\mathcal{Y}_{n,n,n}|+4)|w_{0,n}|}(B)=\mathcal{L}_{(4|\mathcal{Y}_{n,n,n}|+4)|w_{0,n}|}(\varphi(X_n)). 
$$ 
But $\varphi(X_n)\in\mathcal{Y}_{n,n,n}$ and so $\mathcal{L}_{(4|\mathcal{Y}_{n,n,n}|+4)|w_{0,n}|}(\varphi(X_n))$ is the language of an element of $\mathcal{Y}_{n,n,n}$, meaning 
that $\mathcal{L}_{(4|\mathcal{Y}_{n,n,n}|+4)|w_{0,n}|}(B)$ is also the language of an element of $\mathcal{Y}_{n,n,n}$.  
However, this contradicts Equation~\eqref{eq:contra}, and so no such shift $X$ or factor map $\varphi$ can exist.
\end{proof}

We note that the sequence  $(a_n)_{n\in\N}$ in Proposition \ref{prop:nonFactor} is computable. For instance, from the proof it follows that one may define this sequence as
$$ 
a_n:=2\cdot(4|\mathcal{Y}_{n,n,n}|+4)|w_{0,n}|, 
$$ 
where $|\mathcal{Y}_{n,n,n}|$ is the number of sofic shifts on an $n$ letter alphabet that can be written as range $n$ block codes of a shift of finite type, also on an 
$n$ letter alphabet (and thus is less than $2^{{n}^{2n+1}}$) and defined with minimal forbidden words of length at most $n$. The number $|w_{0,n}|$ is recursively defined and depends on $|\mathcal{Y}_{r,r,r}|$ for $r\leq n$.

We conclude this section by showing that $\mathcal{Z}(a_n)$ is a large set, in a sense made precise in the next proposition.  The result of the proposition, with $\mathcal{Z}(a_n)$ replaced with the set of all language stable shifts, appeared in~\cite[Corollary 5.2]{CK21} with essentially the same proof.  As the proof is short and needs a small amount of adapting to apply to $\mathcal{Z}(a_n)$, we include it here for completeness.

\begin{proposition}\label{prop:g-delta}
Let $(a_n)_{n\in\N}$ be a sequence of positive integers and fix $a\in\N$.  For any $0\leq h\leq\log(a)$, the set 
$$ 
T_h:=\{Y\in\mathcal{Z}(a_n)\colon\mathcal{L}_1(Y)\subseteq\{0,1,\dots,a-1\} \text{ and } h_{top}(Y)\geq h\} 
$$
is a dense $G_{\delta}$ subset, with respect to the Hausdorff metric, of the space of all subshifts of $\{0,1,\dots,a-1\}^{\Z}$ that have entropy at least $h$.
\end{proposition}
\begin{proof}
Fix $0\leq h\leq\log(a)$.  Let $S\subseteq\{0,1,\dots,a-1\}^{\Z}$ be a shift with entropy at least $h$.  Let $(S_k)_{k\in\N}$ be the SFT cover of $S$.  Note that $S_k\in\mathcal{Z}(a_n)$ for all $k$ since $S_k$ has only finitely many minimal forbidden words.  Moreover, $h_{top}(S_k)\geq h_{top}(S)$ for all $k$.  Finally, with respect to the Hausdorff metric, $d(S,S_k)\leq2^{-k}$ because $\mathcal{L}_k(S)=\mathcal{L}_k(S_k)$.  Therefore $S$ is a limit point of the set of all subshifts of $\{0,1,\dots,a-1\}^{\Z}$ that have entropy at least $h$.  It follows that $\overline{T_h}$ is the set of all subshifts of $\{0,1,\dots,a-1\}^{\Z}$ that have entropy at least $h$.

We are left with checking that $T_h$ is a $G_{\delta}$ set.  Note that if $T_0$ is a $G_{\delta}$ subset of the space of all subshifts of $\{0,1,\dots,a-1\}^{\Z}$, 
then $T_h$ is a $G_{\delta}$ subset of the space of all subshifts of $\{0,1,\dots,a-1\}^{\Z}$ that have entropy at least $h$.  Thus it suffices to prove the claim for $T_0$.  To do so, recall that $\mathcal{X}_{a,n}$ denotes the set of shifts of finite type on the alphabet $\{0,1,\dots,a-1\}^{\Z}$ that can be defined using forbidden words of length at most $n$.  For each $X\in\mathcal{X}_{a,n}$, let 
$$ 
\mathcal{U}(X):=\{Y\subseteq\{0,1,\dots,a-1\}^{\Z}\colon\mathcal{L}_{n+a_n}(Y)=\mathcal{L}_{n+a_n}(X)\}, 
$$ 
which is open with respect to the Hausdorff metric.  Then we have that 
$$ 
T_0=\bigcap_{n=1}^{\infty}\left(\bigcup_{m=n}^{\infty}\bigcup_{X\in\mathcal{X}_{a,m}}\mathcal{U}(X)\right)
$$ 
and so $T_0$ is a $G_{\delta}$ set.
\end{proof}

\subsection{Questions about well approximable shifts }
\label{sec:questions-on-WA}

It is natural to expect that there are some effective bounds for which language stable shifts are well approximable. 
\begin{question}
Are there effective bounds on the function $A_{(n,n,n,n)}$ from Definition~\ref{def:const}?
\end{question}

It is shown in~\cite{CK21} that every language stable shift has a characteristic measure that is a measure of maximal entropy, and in Theorem~\ref{th:char-periodic} we showed that the same holds for any symbolic factor of a  well-approximable language stable shifts.  However, we do not know how general this result is, and so we ask the following question.

\begin{question}
Does every subshift factor of a language stable shift have a characteristic measure of maximal entropy?
\end{question}
More generally, we can ask if the same holds for any shift, but we still do not know if every shift even has a characteristic measure. This brings us to a related question about construction of characteristic measures. 

\begin{question}
Assume that  $X$ is a  mixing subshift and let $\varphi\in\Aut(X)$ be an automorphism.  Let $(X_n)_{n\in\N}$ be the SFT cover of $X$ and let $\mu_n$ be the unique measure of maximal entropy on $X_n$.  If $\mu$ is a weak* limit point of $(\mu_n)_{n=1}^{\infty}$, is $\varphi_*\mu=\mu$?
\end{question}
Note that this question is particularly relevant for shifts that are not necessarily language stable.

\section{Largeness of the class of language stable shifts}
\label{sec:large-LSS}
\subsection{Language stable shifts are closed under bounded speedups}
While language stable shifts are not closed under passage to factors, they are closed under other operations.  In particular, we show that language stability is preserved by speeding up the transformation.

The classic speedup is given by taking a power of the transformation, for example $(X, \sigma^n)$ is a bounded speedup of $(X,\sigma)$ for any $n\in\N$.  We consider the more general setting and  
for a shift $(X, \sigma)$, define a {\em bounded speedup} is a self-homeomorphism $S$ of $X$ of the form 
$$ S(x) = \sigma^{\rho(x)}(x)$$
for some bounded function $\rho \colon X \to \N^*$.

The power map $\rho$ in a bounded speed in an aperiodic system must be continuous (see~\cite[Proposition 2.2]{AlvinDrewOrmes}), and it is known that bounded speedups preserve many properties of the initial system.  For example, a bounded speedup of a subshift is expansive, and hence is also a subshift, and similarly the bounded speedup of a substitution subshift is also a substitution (see~\cite{AlvinDrewOrmes} for these results and more background on such systems).  We give a result in this spirit in Theorem~\ref{th:speedup}: a bounded speedup of a language stable shift, and more generally of an induced map on a clopen set of a minimal language stable shift, is itself language stable.

We start with an elementary lemma used to bound the lengths of minimal forbidden words. 
\begin{lemma}\label{lem:useful}
Let $X$ be a shift. If $u,v,s$ are words such that $uv, vs\in\CL(X)$ but $uvs$ is a forbidden word in $X$, then the word  $uvs$ contains a minimal forbidden word of length at least $|v|$. 
\end{lemma}
\begin{proof}
If $f$ is a subword of $uvs$ that is of minimal length and is forbidden in $X$ and contains $v$ as a subword, then $f$ is a minimal forbidden word. 
\end{proof}

\begin{theorem}\label{th:speedup}
Let $X \subset  A^\Z$ be a shift and let $\rho \colon  U \to \N^*$ be a   continuous  function such that the map  $S(\cdot) = \sigma^{\rho( \cdot)} (\cdot)$ is a homeomorphism on a clopen set $U\subset X$. 
Then  $(X, \sigma)$ is language stable if and only if   $(U, S)$ is. 
\end{theorem}

 It follows that for a minimal language stable shift, any induced map $ \tau(\cdot) =  \sigma^{\inf \{n >0\colon  \sigma^n(\cdot) \in U \}} (\cdot)$ on a clopen set $U$ is language stable.  Furthermore, using an argument similar to that in~\cite[Proposition 2.2] {AlvinDrewOrmes}, we can relax the hypothesis of continuity on $\rho$ to only require that  it is bounded on an aperiodic subshift. 

\begin{proof}
To simplify notation, we make a small abuse of notation and let $\rho \colon \Z \times U \to \Z$ denote the cocycle associated to $\rho$, meaning that 
\[
\rho(n, x) = \begin{cases}  \sum_{i=0}^{n-1} \rho(S^i x) & \textrm{ for } n>0\\
0 & \textrm{ for } n=0\\
 -\sum_{i=0}^{-n-1} \rho(S^{i+n} x)& \textrm{ for } n<0.
\end{cases} 
\]
Hence for any integer $n$ and $x\in X$, we have  that $S^n x= \sigma^{\rho(n,x)}(x)$. 

Since $\rho$ is a continuous function on $U\subset X$, it is locally constant. Choose an integer  $N> \sup_{x\in U} |\rho(x)|$ such that $\rho$ is constant on any cylinder defined by a word of length $N$. 
To simplify notation, for the remainder of this proof we  let $J(N)$ denote the interval $[-N,N]$ and 
for $q\in\Z$ we write $q+J(N)$ for the interval $[-N+q, N+q]$, and use analogous notation for translates of intervals $I, J$.
We claim that for sufficiently large $N$, the map from $\phi\colon U\to (A^{2N+1})^\Z$ given by
\begin{equation}
    \label{def:phi}
\phi\bigl((x_n)_{n\in\Z}\bigr) = (x_{\rho(n, x) + J(N)})_{n \in \Z}
\end{equation}
defines a conjugacy between the induced system $(U,S)$ and the shift $(X, \sigma)$. 
To check this, note that since the sequence  $(\rho(n,x))_{n \in \Z}$ is increasing and the distance between two consecutive terms is uniformly bounded, the map is a homeomorphism for any sufficiently large $N$.  Furthermore, a  direct computation shows that the map $\phi$ intertwines the action of $S(\cdot) = \sigma^{\rho(\cdot)} (\cdot)$ and the shift map.
Let $(Y,\sigma)$ denote the image of the shift $(X,\sigma)$ under $\phi$.

Assume first that  $(X,\sigma)$ is language stable. Let  $\tf=  \tf_0 \dots \tf_{\ell-1}$ be a minimal forbidden word for $Y$ of  length $\ell$, where $\ell$ satisfies
$\inf_{x \in U}  \rho(\ell, x)  > 2N$. This means that there exist $x, y\in U$ such that  
$$
\begin{aligned}
x_{J(N)} = \tf_0&  & &  & y_{J(N)} = \tf_1 \\ 
x_{\rho(x)+ J(N)} = \tf_1 & & & & y_{\rho(y)+ J(N)} = \tf_2 \\
  \vdots \ \ \  \ &   & & &  \vdots  \ \ \ \ \\
x_{\rho(\ell-2, x)+ J(N)} = \tf_{\ell-2} &  & & & y_{\rho(\ell-2, y)+ J(N)} = \tf_{\ell-1}\\
x_{\rho(\ell-1, x)+ J(N)} \neq \tf_{\ell-1} & & & &  y_{\rho(-1, y)+ J(N)} \neq \tf_{0}.
\end{aligned}
$$

For $i =0,1$ and  $j=\ell-1, \ell-2, \ell-3$, let $I_{i,j}(x) $ denote the interval  $I_{i,j}(x)  = \bigcup_{k=i}^{j} \rho(k,x) +J(N)$, 
and define the interval $I_{i,j}(y) $ analogously.
Note that  $\rho(i, x) = \rho(i-1,y) +\rho(x)$ for each $0<i < \ell-1$ and  $x_{I_{1,\ell-2}(x)}  = y_{I_{0,\ell-3}(y)}$. Moreover the interval $J(N)$  is not a subset of  $ I_{1,\ell-2}(x)$, because $\rho$ takes only positive values. 
Similarly the interval $\rho(\ell-1, x)+J(N)$  is not a subset of  $ I_{1, \ell-2}(x)$.
Let $I$ denote the interval $J(N) \setminus I_{1,\ell-2}(x),$ and let $J$ denote the interval $ (\rho(\ell-1, x)+J(N)) \setminus  I_{1,\ell-2}(x)$. 
Since $\ell$ is sufficiently large, the intervals $I$ and $J$ are disjoint.

Choose $w$ to be the word such that
$$ 
\begin{aligned}
w_I &= x_I, \\
w_{I_{1,\ell-2}(x)} &= x_{I_{1,\ell-2}(x)},\\
w_J  &= y_{J+ \rho(x) }. 
\end{aligned}
$$
By construction, we have that $w_{\rho(i,x) + [-N,N]} = \tf_i$ for each $0 \le i < \ell$. Since $\tf$ is forbidden in $Y$,  the word $w$ is forbidden in $X$. It follows from   Lemma~\ref{lem:useful} that $w$ contains a minimal forbidden word  for $X$ of length $f(\ell)$  between 
$|I_{1,\ell-2}(x)| = 2N + \rho(\ell- 2,x) -  \rho(x) +1$ 
and $|I_{0,\ell-1}(x)| = 2N +  \rho(\ell-1,x)  +1$.  
Thus  for two distinct lengths $\ell_1 <  \ell_2$ of minimal forbidden words, we have that 
$$\begin{aligned}
    f(\ell_2)- f(\ell_1)  &\le \rho(\ell_2-1, x)  - (\rho(\ell_1-2, x) - \rho(x) ) \\ 
                        &= \sum_{i=\ell_1 -2}^{\ell_2-2} \rho(S^i x)  + \rho(x)
                        \le (\ell_2  - \ell_1 +2) \max_{x\in U} \rho(x).  
 \end{aligned}
$$

It follows that when there are arbitrary large gaps between consecutive lengths of minimal forbidden words in $X$, the same holds for $Y$. In particular, if $(X, \sigma)$ is language stable then so is $(Y, \sigma)$, and thus so is the induced map $(U, S)$.

Conversely, assume that $(U,S)$ is language stable and thus so is the system $(Y, \sigma)$, where again $(Y, \sigma)$ is the system defined to be the image of the $(X, \sigma)$ under the map $\phi$ defined in~\eqref{def:phi}. 
Let $w$ be a minimal forbidden word for $X$ of length $\ell$, with this length  to be determined.
For $x\in U$  and integers $i<j$, let $I_{i,j}(x)$ denote the interval  $\bigcup_{k=i}^{j} \rho(k,x) +J(N)$. Define $n_\ell$  to be the integer 
$$ n_\ell = \sup\{n \ge 0\colon  \text{ there exists }  x \in U \text{ such that }  x_{I_{0, n}(x)} \textrm{ is a subword of } w  \}.
$$

By definition, for any $\ell > 2N+1$, the set in this definition is nonempty.
Let $x^*\in U$  be some point associated to $n_\ell$. 
As any subword of $x^*$ has shorter length, it follows that 
\begin{align}\label{eq:encadre}
2N+1 + n_\ell \min_{x \in U} \rho(x)   \le |I_{0,n_\ell}(x^*)|    \le \ell-1.
\end{align}
Since $n_\ell$ is maximal, it also follows that 
\begin{align}\label{eq:nell}
 \ell  <  (n_\ell +1)  \max_{x \in U} \rho(x) +2N+1. 
\end{align}

Since $n_\ell$ is maximal, neither  $x^*_{I_{-1, n_\ell}(x^*)}$ or  $x^*_{I_{0, n_{\ell}+1 }(x^*)}$ can be a subword of $w$. It follows from  inequalities~\eqref{eq:encadre} and~\eqref{eq:nell} that for any sufficiently large $\ell$, the intervals $I_{-1, n_\ell}(x^*) \setminus I_{0, n_\ell}(x^*)$   and  $I_{0, n_\ell+1}(x^*) \setminus I_{0, n_\ell}(x^*)$ are disjoint.

 Let $\tu, \ts \in A^*$ be two words 
 such that  $w = \tu x^*_{I_{0, n_\ell}(x^*)}\ts$.
The length of $I_{0,n_\ell +1}(x^*)$ is greater than $|x^*_{I_{0, n_\ell}(x^*)}\ts |$, as otherwise  we can find $x \in U$ such that $x_{I_{0, n_\ell +1}(x)}$ is a subword of  $w$, a contradiction of the maximality of $n_\ell$. Similarly the length of $I_{-1,n_\ell}(x^*)$ is greater than $|\tu x^*_{I_{0, n_\ell}(x^*)}|$.

 Since $w$ is forbidden,  the words $\ts$ and  $\tu$ can not be both empty. Since the words $\tu x^*_{I_{0, n_\ell}(x^*)}$ and $x^*_{I_{0, n_\ell}(x^*)}\ts $ 
 are allowable words in the language of $X$, there exist words   $\tu^-$ and $\ts^+$ in $A^*$ such that the words $\tu^- \tu x^*_{I_{0, n_\ell}(x^*)}$ and $x^*_{I_{0, n_\ell}(x^*)}\ts \ts^+$ are allowed in $X$ and 
 $|\tu^- w \ts^+|= |I_{-1, n_\ell+1}(x^*)|$. Finally set  $w_s$ (respectively, $w_u$) to be the suffix (respectively, prefix) of length $2N+1$ of the word  $\tu^- w \ts^+$.
Let  $\tf = \tf_{-1} \cdots \tf_{n_\ell +1}$ be the word  in $(A^{2N+1})^{n_\ell +3}$ defined by 
$$ 
\begin{aligned}
& \tf_i = x^*_{\rho(i,x^*) + J(N)} & \textrm{ for } 0 \le i \le  n_\ell,\\
& \tf_{n_\ell+1} = w_s,  & \textrm{ and }  \tf_{-1} = w_u.
\end{aligned}
$$
This defines a forbidden word for $Y$, as otherwise  the word $\tu^- w \ts^+$ occurs $X$. By construction, the word  $\tf_0 \dots  \tf_{n_\ell}$ is allowed  in $Y$.  Because the words $\tu^- \tu x^*_{I_{0, n_\ell}(x^*)}$ and   $ x^*_{I_{0, n_\ell}(x^*)} \ts \ts^+$ are allowed in $X$, it is also straightforward to check that the words  $\tf_{-1} \cdots \tf_{n_\ell}$ and $\tf_0 \cdots \tf_{n_\ell+1}$ are allowed in $U$. 
This means that  $\tf$ is  a minimal forbidden word of $Y$ with length  $g(\ell)= n_\ell+3$.

By Inequalities~\eqref{eq:encadre} and~\eqref{eq:nell}, as in the first part of the proof, there are arbitrary gaps between consecutive lengths of minimal forbidden words of $X$ when this holds for $Y$, and language stability follows. 
\end{proof}

\subsection{Many language stable shifts are $\beta$-shifts}
We describe how the class of language subshifts interacts with a well-known family parameterized by the reals,  the $\beta$-shifts. 
We start with a brief summary of the properties of $\beta$-shifts, and
refer to the survey~\cite{Blanchard} for further background. 

 Let $\beta>1$ be real number. Set 
 $$ d(1,\beta) = (\lfloor \beta T_\beta^{n-1}(1)\rfloor)_{n\ge 1}, \quad \textrm{ where } T_\beta(x) = \beta x -\lfloor\beta x\rfloor.$$ 
When  the expansion $d(1, \beta)$ is {\em finite}, meaning that   $d(1,\beta) =d_{0} \dots d_{k} 0 0 \dots$ for some finitely many nonzero integers $d_0, \dots, d_k$,  set  $$d^{*}(1, \beta) = d_{0} \cdots d_{k-1} (d_{k}-1)d_{0} \cdots d_{k-1} (d_{k}-1) \cdots$$
to be the periodic expansion.  For the remainder of this section, given any finite expansion we  consider the associated infinite periodic one $d^{*}(1,\beta)$ instead of $d(1,\beta)$. 
Note that $d(1, \beta)$ starts with  $\lfloor\beta\rfloor$, the greatest letter of the alphabet.  

Letting $<_{\textrm{lex}}$ be the lexicographic ordering, 
the expansions $d(1, \beta)$ are characterized in  the following way:
\begin{proposition}[see, for example, \cite{Blanchard}]\label{prop:carExpansion}
A sequence $d \in \{0, \ldots, \lfloor\beta\rfloor \}^{\N}$ is the expansion of some number $\beta>1$ if and only if
$$ \sigma^{s} d <_{\textrm{lex}} d, \quad \text{ for all } s\ge 1.$$  
\end{proposition} 
The associated subshift $S_{\beta}$, called a $\beta$-{\em shift},  is defined by
$$S_{\beta} = \{ (s_{n})_{n \in \Z} \in\{0, \ldots, \lfloor\beta\rfloor \}^{\Z}: s_{n} \cdots s_{m} \le_{\textrm{lex}} d(1,\beta) \quad  \text{ for all } n < m \}.$$   
Another equivalent way to define this subshift  is by taking bi-infinite edge paths
in the graph whose edges are labeled by the integers $0, 1, 2, \ldots$, with an edge from the vertex $i$ to $i+1$  labelled by $d_{i}$ and when $d_{i}>0$, there are edges between the vertices $i$ to $0$ labelled by $d_{i}-1$, $\ldots$, $0$. 

The arithmetic properties of the real $\beta$ and the dynamical properties of the shift $S_{\beta}$ are connected.  For example, the class of $\beta$ such that  $S_{\beta}$ is sofic contains  all Pisot numbers and is a subset of the Perron numbers.  Surprisingly,  $S_{\beta}$ is a shift of finite type for  $(1+ \sqrt{5})/2$ but not for $(3+ \sqrt{5})/2$. 
Furthermore, the shifts $S_{\beta}$ depend continuously on the real parameter $\beta$ in the Hausdorff topology. It follows that 
for a generic set of parameters $\beta$, the subshift $S_\beta$ is language stable, 
We give a more precise characterization of which $\beta$-shifts are language stable. To do so, we consider the combinatorial properties of the associated expansion sequences, making use of a characterization of the associated language.
\begin{proposition}[see, for example,  \cite{Blanchard}]\label{prop:codedsystem} The language of a $\beta$-shift $S_{\beta}$ is defined by concatenations of all words in the set 
$$Y= \{wb: w \textrm{ is a prefix of } d(1,\beta) , b < d_{|w|}     \}.$$  \end{proposition}

We make use of this to describe the minimal forbidden words of a $\beta$-shift. 
\begin{lemma}\label{lem:CarMinimalWordBetaShift} For the $\beta$-shift $S_\beta$, the set of the minimal forbidden words  $\CM(S_{\beta})$
is the set 
$$  \{ wb:  w \textrm{ is a prefix of } d(1,\beta),  b > d_{|w|}, \text{ for any strict suffix } w' \textrm{ of } w, w'b\le d(1,\beta)  \}.$$  
\end{lemma}   
\begin{proof}
By definition of the subshift $S_\beta$, any word  of the form $wb$ with  $w$ a prefix of $d(1, \beta)$,   $b > d_{|w|}$  is 
 forbidden. If $w'b \le  d(1,\beta)$ for any strict suffix of $w$, then $wb$ is minimal forbidden.
 
 Conversely, assume that $u \in \CL(S_{\beta})$ is a word such that  $ub \not\in \CL(S_{\beta})$ is a minimal forbidden word for some letter $b$. 
 By Proposition~\ref{prop:codedsystem}, it follows that $u$ can be written as a concatenation of the form 
 $$u= w_{0} b_{{0}} w_{1}b_{{1}} \cdots w_{\ell}  b_{\ell}$$
where  $w_{i}$ are prefixes of  $d(1,\beta)$, $b_{n_{i}} < d_{|w_{i}|}$, and    $b_{\ell}\le d_{|w_{\ell}|}$. Since $ub$ is not an allowable word and $d_{0} = \lfloor \beta \rfloor \ge b$, these imply that $b_{\ell} = d_{|w_{\ell}|}$ and $b> d_{|w_{\ell}|+1 }$.  It follows that $w_{\ell} b_{{\ell}}b \not\in \CL(S_{\beta})$.   Moreover since any strict suffix of $ub$ is allowable, we obtain the desired inclusion.       
\end{proof}

 We give a sufficient condition to ensure the existence of a minimal forbidden word  with certain bounds on the length. 
 \begin{lemma}\label{lem:MFWordGap}  Assume that  any  $i \in \{  i_1 +1,, i_1+2, \ldots, i_{2} \}$, satisfies  $d_{i} < d_{0}$.  
Then   for any $i_{1} < n < i_{2}$, there exists a minimal forbidden word of length between $n-i_{1}$ and $n+2$.   
\end{lemma}   
\begin{proof}
Consider the word $d_0\cdots d_n\in\CL(S_\beta)$.
Since $d_{n+1} < d_0$, the word 
$d_0 \cdots d_n d_0$ is forbidden. However it follows from Proposition~\ref{prop:codedsystem} that   $d_{i_1+1} \cdots d_n d_0\in\CL(S_\beta)$, as each $d_j < d_0$ for $i_1< j\le n$.    
Thus it follows from Proposition~\ref{lem:useful} that  the word $d_0 \cdots d_n d_0$ contains a minimal forbidden word of length at least  $n-i_1$.
\end{proof}

\begin{corollary}\label{cor:CNLSSBetaShift}
For a $\beta$-shift, if  the set $\{ i \in \N:   d_{i} = d_{0}\}$ is finite, 
then the set of lengths of minimal forbidden words is relatively dense.  

In particular the subshift $S_\beta$ is not language stable. 
\end{corollary}

\begin{proof} 
Assume that for any index $i$ greater than $D$, $d_i < d_0$. Lemma~\ref{lem:MFWordGap} ensures there exists a length of a minimal forbidden word at distance at most $\max (D, 2)$ of  any  integer $n>D$.   
\end{proof}

It follows from Corollary~\ref{cor:CNLSSBetaShift} that we can restrict our focus to expansion sequences of $\beta$-shifts where the letter $d_0$ occurs  infinitely many times.

\begin{lemma}\label{lem:GapLength}
For a $\beta$-shift, if some nonempty word  $w$  is the prefix and suffix of some word  $d_{0} \cdots d_{n}$, 
then there is no minimal forbidden word  of length  $n+1$. 
\end{lemma}   
\begin{proof}
The word $w$ is of the form  $w = d_0 \cdots d_\ell$. By contradiction, assume there is a minimal forbidden word of length $n+1$. By Lemma~\ref{lem:CarMinimalWordBetaShift}, this word has the form $d_0 d_1 \cdots d_{n-1} b$ for some$ b > d_n =d_\ell$. Since any strict suffix of this word is allowable, the word $ d_0 d_1 \cdots d_{\ell-1}b$ is also allowable. 
But this is a contradiction, since $b >d_\ell$. 
\end{proof}

When  $d_{0} \cdots d_{n-1} = w \cdots w$, the words $w$ is said to {\em occur} in $d(1, \beta)$ and the index $n- |w|$ is  called an {\em occurence} of $w$. 
In this case, any prefix of $w$ also occurs in $d(1, \beta)$, and so we can apply  Lemma~\ref{lem:GapLength} for each length of the prefix $w$.  This means that there is no minimal forbidden word of length $n-(i-1)$ for any $ 0 \le i < |w|$, 
providing gaps in the  complement of  lengths of minimal forbidden words. 
Summarizing this, we have the following corollary: 
   
\begin{corollary}\label{cor:BetaShiftLSS}
For a $\beta$-shift, the complement of the set of lengths of minimal forbidden words  contains the set
$$ \bigcup_{w \textrm{ prefix of } d(1,\beta)} \{n + |w|:  n \textrm{ is an occurrence  of } w \textrm{ in } d(1, \beta) \}. $$
In particular, if  every prefix  of $d(1,\beta)$ occurs at least twice in $d(1,\beta)$, then $S_{\beta}$ is language stable. 
\end{corollary}

Schmeling~\cite{Schmeling} shows that the sequence $d(1, \beta)$ is  generic for the measure of  maximal entropy in $S_\beta$ for  Lebesgue almost every $\beta\in\R$. In particular, any  prefix of  $d(1, \beta)$  occurs in  $d(1, \beta)$ on a set of indices of positive density. Combining this result with Corollary~\ref{cor:BetaShiftLSS}, we conclude that the set of $\beta$-shifts is large. 
\begin{corollary}\label{cor:BetaShiftLSSgeneric}
The set 
$$\{\beta >1: S_\beta  \textrm{ is language stable } \}$$
is a dense $G_\delta$ in $(1, +\infty)$ of full Lebesgue measure. 
\end{corollary}

Hofbauer~\cite{Hofbauer} and  Walters~\cite{Walters} show that any $\beta$-shift admits a unique measure of maximal entropy, and hence has a characteristic measure. Moreover, Climenhaga and Thompson~ \cite{ClimenhagaThompson} show 
that the same holds for any factor of a $\beta$-shift, and so any factor of a $\beta$-shift also has a characteristic measure.

We conclude with a particular example that goes beyond shifts of finite type (see~\cite{ClimenhagaThompson} for the precise definitions).
\begin{example}[A specified/synchronized $\beta$-shift that is language stable]\label{ex:BetaShift}
We define an expansion sequence by induction using 2 sequences of words.
Consider the alphabet $\{0,1,2 \}$ and set 
$w_0 = 222$ and $u_0=111$. 
Define $w_{n+1} = w_{n}u_n w_n$ and fix some word $u_{n+1} \in \{0,1\}^*$ that is smaller (in lexicogarphic order) than $u_nw_n$. In particular, the length of this word is smaller than that of  $u_nw_n$.
Taking the limit of the worsd $w_n$, we obtain an infinite word $w$ in $\{0,1,2\}^\N$   

By construction, the sequence $w$ satisfies $ \sigma^s w <_{\textrm{lex}}w$ for any $s \ge 1$. Hence, by Proposition~\ref{prop:carExpansion}, the sequence $w$ is the expansion of some number $\beta>1$.    
Moreover, by construction, any prefix of $w$ occurs infinitely many times. It follows from Corollary~\ref{cor:BetaShiftLSS}, the associated subshift $S_\beta$ is language stable.  

Additionally, at each step $n$, we can choose $u_n$ to be a power of $1$. Then no $0$ occurs in the sequence and so the subshift $S_\beta$ is specified (see~\cite{Bertrand} or~\cite{Blanchard}). In particular if the sequence of powers of $1$ is increasing, then the sequence $w$ is not ultimately periodic and the subshift $S_\beta$ is not sofic. 

Similarly, if  $u_n$ is taken to be an increasing sequence of strings of $0$'s, the word $111$ never occurs in $w$ but is allowable in the subshift. Then the subshift $S_\beta$ is synchronizing but not specified (again, see~\cite{Bertrand} or~\cite{Blanchard}).
\end{example} 

\section{Aperiodic linear complexity shifts are language stable}
\label{sec:linear}

We start with a general result on sufficiently low complexity shifts that is useful for checking language stability. 
\begin{theorem}\label{theo:NScompLSS} 
Assume that  $(X, \sigma)$ is an aperiodic shift. 
\begin{enumerate}
    \item 
    \label{item:one-linear}
    If the complexity $p_X$ of $(X, \sigma)$ satisfies $p_X(n) = O(n)$, then the upper uniform density of the set of lengths of the  bispecial words is zero. 
    \item
        \label{item:two-linear}
If the complexity of the shift $(X, \sigma)$ satisfies 
$$\liminf_{n\to\infty} \bigl(p_X(n+1)-p_X (n) \bigr)< \infty,$$ then the lower density of the set of lengths of bispecial words is zero. 
\end{enumerate} 
\end{theorem}
\begin{proof}
To prove Part~\eqref{item:one-linear}, 
by Cassaigne's Theorem~\cite{Cas96}, a shift $(X, \sigma)$ satisfies $p_X(n) = O(n)$ if and only if there exists some $K\in\N$ such that $p_X(n+1)-p_X(n) \leq K$ for all $n\in\N$.  
It follows that for each $n\in\N$, there are at most $K$ right special (respectively, left special) words of length $n$.  

We proceed by contradiction. 
Assume that  the set $F$ of  lengths of the bispecial words  of $X$ has upper uniform density $\alpha$ for some $\alpha > 0$. 
Choose $k\in\N$ such that $k \alpha/2 > K^{2}$.   
Since  $X$ is aperiodic, there is some power $p>0$ such that  for any word $u$ of length at most  $k$, the word  $u^{p}$ is forbidden. 
By assumption, we can choose sufficiently large $n\in\N$, in particular greater than  $(p+1)  k $,   such that the interval $[n+1, n+k]$ contains at least $k \alpha/2$ different lengths of bispecial words.  
Let $\tf_{1}, \ldots, \tf_{\lfloor k \alpha/2\rfloor}$  denote these bispecial words.

By construction, each prefix  and  each suffix of each $\tf_{i}$ is a  special word.  Thus by the Pigeonhole Principle, 
there are two indices  $ 0<i<j \leq  K^{2} +1$ such that $\tf_{i}$ and $\tf_{j}$ share the same prefix $\tp$ and share the same suffix  $\ts$ of length $n$.   
It follows that  the suffixes  of $\tp$ of lengths $ 2n- |\tf_{i}|$  and  $ 2n- |\tf_{i}|$  are both prefixes    of  $\ts$. Thus, by the Fine-Wilf Theorem~\cite{FW}, the largest of these prefixes, of length at least $n-k$, is of the form $\tu^{\ell} \tu'$ where $\tu$ and $\tu'$ are words such that 
$|\tu| = ||\tf_{j}|-|\tf_{i}|| (\le k ) $,  
$|\tu'| <  ||\tf_{j}|-|\tf_{i}||$, and 
$$\ell \ge \lfloor (n-k) /||\tf_{j}|-|\tf_{i}|| \rfloor \ge (n-k)/k.$$    
However, this contradicts the choice of $n$.

By the assumption in Part~\eqref{item:two-linear}, 
the set $A= \{n \colon   p_{X}(n+1)-p_{X}(n) \le  K \}$ is infinite for some $K\in\N$. Hence  there are at most $K$ right special  words  of length $n \in A$, and the same bound holds for the number of left special words.    Again, we proceed by contradiction, assuming that  the set $F$ of  lengths of the bispecial words  of $X$ has lower density $\alpha$ for some $\alpha > 0$. 
We choose $k\in\N$ as in Part~\eqref{item:one-linear}, 
and note that again there exist arbitrarily large integers $n \in A$, and in particular larger than $(p+1)  k$, such that the interval $[n, n+k]$ contains at least $k \alpha/2$ different lengths of bispecial words.
We then continue exactly as in Part~\eqref{item:one-linear}. 
\end{proof}

\begin{corollary}\label{cor:LinearLSS}   
Any aperiodic shift with  complexity satisfying 
$$\liminf_{n\to\infty} \left(p_X(n+1)-p_X (n) \right)< \infty,$$ is language stable. 
\end{corollary} 
It follows immediately  that any aperiodic shift with linear complexity is language stable.

\begin{proof}
A minimal forbidden word has the form $awb$ for some letters $a,b$ in the alphabet and some bispecial word $w$ in the language.  By Theorem~\ref{theo:NScompLSS}, the set of lengths of bispecial words has lower density zero, and language stability follows. 
\end{proof}

We note that the assumption of aperiodicity in Corollary~\ref{cor:LinearLSS} is necessary, as can be seen by  considering 
the closure under the shift of the sequence with a single $1$ and all other entries $0$ (note that this system contains the fixed point of all $0$s).  
The complexity of this system satisfies $p(n) = n+1$ for all $n\in\N$, while the words $1 0^m1$ are minimal forbidden words for all $m\in\N$ and so the system is not language stable.

\begin{example}[An aperiodic non-language stable shift with subquadratic complexity] 
Let $\tau$ be the substitution defined by $0 \mapsto 010$ and $ 1 \mapsto 11$, 
and let $(X, \sigma)$ be the associated shift defined by $x\in X$ if every word in $x$ is a subword of $\tau^n(0)$ for some $n\ge 0$. 
As we consider the fixed point of $0$, this shift is aperiodic, 
and the complexity of the shift is approximately 
$n \log_2 \log_2 n$. 
Cassaigne~\cite[Section 4.4]{Cas97} shows that for any $\ell$ with $\ell \ge 3$ not a power of $2$, 
this shift contains a minimal forbidden word of length $\ell$ (precisely a non-strict bispecial word) and so this shift is not language stable.  
\end{example} 
We note that the shift defined in this example is not minimal.  
This leads to a natural question, namely if every minimal shift is language stable.  However, we believe that this is unlikely to hold, and so we ask: 

\begin{question}
\label{question:complexity-theshold}
Assuming that not every minimal shift is language stable, what is the complexity threshold for the existence of a minimal shift that is not language stable? 
\end{question}

\end{document}